\documentclass[a4paper,11pt]{article}
\usepackage{amsmath, amsfonts, amscd, amssymb, amsthm, enumerate, xypic, psfrag}

\def\bfB{\mathbf{B}}

\def\reg{\text{reg}}
\def\red{\text{red}}
\def\nil{\text{lnil}}

\def\Aprop{\text{A}}

\newcommand{\defi}{\operatorname{def}}
\newcommand{\card}{\operatorname{card}}
\newcommand{\Hom}{\operatorname{Hom}}

\newcommand{\SUP}{\operatorname{SP}}
\newcommand{\id}{\operatorname{id}}

\newcommand{\Ker}{\operatorname{Ker}}
\newcommand{\PSet}{\operatorname{PSet}}
\newcommand{\Set}{\operatorname{Set}}
\newcommand{\LinCat}{\operatorname{Lin}}
\newcommand{\Real}{\operatorname{Real}}
\newcommand{\Irr}{\operatorname{Irr}}
\newcommand{\Rep}{\operatorname{Rep}}
\newcommand{\CSup}{\operatorname{CSup}}

\newcommand{\Vect}{\operatorname{span}}

\newcommand{\rk}{\operatorname{rk}}

\newcommand{\Ubd}{\operatorname{Ubd}}

\renewcommand{\setminus}{\smallsetminus}
\newcommand{\modu}{\operatorname{mod}}

\def\F{\mathbb{F}}

\def\N{\mathbb{N}}
\def\Z{\mathbb{Z}}

\def\calC{\mathcal{C}}
\def\calD{\mathcal{D}}

\def\calQ{\mathcal{Q}}

\def\calX{\mathcal{X}}
\def\calY{\mathcal{Y}}

\def\lcro{\mathopen{[\![}}
\def\rcro{\mathclose{]\!]}}

\theoremstyle{definition}
\newtheorem{Def}{Definition}[section]
\newtheorem{Not}[Def]{Notation}

\theoremstyle{plain}
\newtheorem{theo}{Theorem}[section]
\newtheorem{prop}[theo]{Proposition}
\newtheorem{cor}[theo]{Corollary}
\newtheorem{lemma}[theo]{Lemma}

\theoremstyle{plain}

\theoremstyle{remark}
\newtheorem{Rems}{Remarks}
\newtheorem{Rem}[Rems]{Remark}
\newtheorem{ex}[Rems]{Example}

\title{Locally finite cycles of linear mappings in countable dimension}
\author{Cl\'ement de Seguins Pazzis\footnote{Universit\'e de Versailles Saint-Quentin-en-Yvelines, Laboratoire de Math\'ematiques
de Versailles, 45 avenue des Etats-Unis, 78035 Versailles cedex, France}
\footnote{e-mail address: dsp.prof@gmail.com}}

\begin{document}

\thispagestyle{plain}

\maketitle

\begin{abstract}
Let $n$ be a positive integer.
An $n$-cycle of linear mappings is an $n$-tuple $(u_1,\dots,u_n)$
of linear maps $u_1 \in \Hom(U_1,U_2),u_2 \in \Hom(U_2,U_3),\dots,u_n \in \Hom(U_n,U_1)$, where $U_1,\dots,U_n$ are
vector spaces over a field. We classify such cycles, up to equivalence, when the spaces $U_1,\dots,U_n$ have countable dimension
and the composite $u_n\circ u_{n-1}\circ \cdots \circ u_1$ is locally finite.

When $n=1$, this problem amounts to classifying the reduced locally nilpotent endomorphisms of a countable-dimensional vector space up to similarity,
and the known solution involves the so-called Kaplansky invariants of $u$. Here, we extend Kaplansky's results to cycles of arbitrary length.
As an application, we prove that if $u_n \circ \cdots \circ u_1$ is locally nilpotent and the $U_i$ spaces have countable dimension, then there are bases $\bfB_1,\dots,\bfB_n$ of $U_1,\dots,U_n$, respectively, such that, for every $i \in \lcro 1,n\rcro$, $u_i$ maps every vector of $\bfB_i$ either to a vector of $\bfB_{i+1}$ or to the zero vector of $U_{i+1}$ (where we convene that $U_{n+1}=U_1$ and $\bfB_{n+1}=\bfB_1$).
\end{abstract}




\section{Introduction}

\subsection{The problem}\label{problemSection}

Throughout, we fix a positive integer $n$. Given $k \in \Z$, we denote by $\overline{k}$ the class of $k$ modulo $n$.
We denote by $\Z/n=\bigl\{\overline{0},\dots,\overline{n-1}\bigr\}$ the group of classes of integers modulo $n$. Given $x \in \Z/n$ and $k \in \Z$, we simply write $x+k$ instead of $x+\overline{k}$, and $x-k$ instead of $x-\overline{k}$.
We denote by $\omega_1$ the first uncountable ordinal, and we set $\infty:=\{\{\emptyset\}\}$ (a set which is not an ordinal).
We convene that $\alpha<\infty$ for every ordinal $\alpha$.
Following the French convention, we denote by $\N$ the set of all non-negative integers, and by $\N^*$ the set of all positive ones.

Let $\F$ be a field. Unless specified otherwise, all the vector spaces under consideration have $\F$ as their ground field.
We denote by $\Irr(\F)$ the set of all monic irreducible polynomials of $\F[t]$.

An \textbf{$n$-cycle} of linear mappings (over the field $\F$) is a family $(U_k,u_k)_{k \in \Z/n}$
in which, for all $k \in \Z/n$, $U_k$ is a vector space (over $\F$), and $u_k$ is a linear mapping from $U_k$ to $U_{k+1}$.
Given such a cycle, we can consider, for each $k \in \Z/n$, the composite linear map
$$(\pi u)_k:=u_{k+n-1} \circ \cdots \circ u_{k+1} \circ
u_{k,}$$
which is an endomorphism of $U_k$.

\begin{center}
\mbox{
\psfrag{0}{$U_0$}
\psfrag{1}{$U_1$}
\psfrag{2}{$U_2$}
\psfrag{dots}{$\cdots$}
\psfrag{n-2}{$U_{n-2}$}
\psfrag{n-1}{$U_{n-1}$}
\psfrag{u0}{$u_0$}
\psfrag{u1}{$u_1$}
\psfrag{u2}{$u_2$}
\psfrag{un-2}{$u_{n-2}$}
\psfrag{un-1}{$u_{n-1}$}
\includegraphics{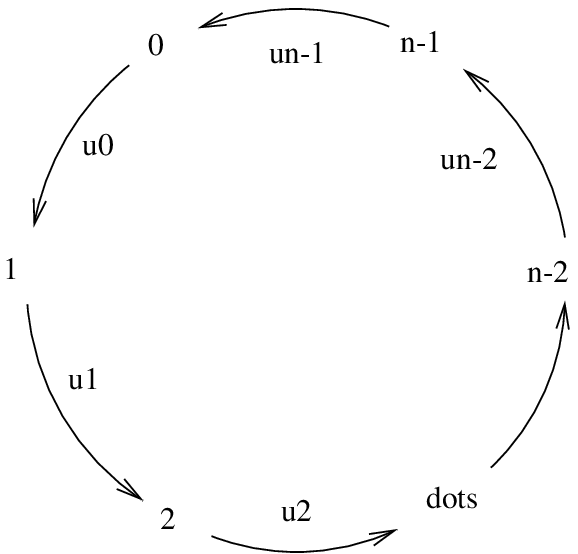}}
\end{center}

The \textbf{dimension} of such a cycle is defined as the sum of the dimensions of the $U_k$ spaces, i.e.\ the greatest such dimension if
one of the $U_k$'s is infinite dimensional.

Two such $n$-cycles $(U_k,u_k)_{k \in \Z/n}$ and $(V_k,v_k)_{k \in \Z/n}$ are called \textbf{equivalent}
when there exists a family $(\varphi_k)_{k \in \Z/n}$ in which:
\begin{itemize}
\item For every $k \in \Z/n$, $\varphi_k$ is an isomorphism from $U_k$ to $V_k$;
\item One has $v_k \circ \varphi_k=\varphi_{k+1}\circ u_k$ for all $k \in \Z/n$.
\end{itemize}
$$\xymatrix{
U_k \ar[r]^{u_k} \ar[d]^{\simeq}_{\varphi_k} & U_{k+1} \ar[d]^{\varphi_{k+1}}_{\simeq} \\
V_k \ar[r]_{v_k}  & V_{k+1.}
}$$
If the two $n$-cycles $(U_k,u_k)_{k \in \Z/n}$ and $(V_k,v_k)_{k \in \Z/n}$
are equivalent, then, for every $k \in \Z/n$, the endomorphism $(\pi u)_k$ is similar to $(\pi v)_k$, i.e.\ there is an isomorphism
$\varphi_k : U_k \overset{\simeq}{\rightarrow} V_k$ such that $(\pi v)_k =\varphi_k \circ (\pi u)_k \circ \varphi_k^{-1}$.
However, the following simple example shows that the converse does not hold:

\begin{ex}
Take $n=2$ and consider the $2$-cycles $u=\bigl((U_0,u_0),(U_1,u_1)\bigr)$ and $v=\bigl((V_0,v_0),(V_1,v_1)\bigr)$
where $U_0=U_1=V_0=V_1=\F$, $u_0=v_1=\id_\F$ and $u_1=v_0=0$. Then, $(\pi u)_k=0=(\pi v)_k$ for all $k \in \Z/2$,
yet $u$ is obviously inequivalent to $v$ because $\rk u_0 \neq \rk v_0$.
\end{ex}

In the finite-dimensional case (i.e.\ when all the spaces $U_i$ are finite-dimensional), the classification
of $n$-cycles up to equivalence is known, see \cite{Nazarova} for the first proof, and the more recent
\cite{DeOliveira,Gelonch,Sergeichuk}.
However, in the infinite-dimensional case there can be no sensible classification: indeed, when $n=1$ the problem is the one of classifying
endomorphisms of vector spaces up to similarity, a problem generally considered as intractable in the infinite-dimensional case.
However, the classification of endomorphisms is known in the countable-dimensional case if one restricts the scope to so-called locally finite endomorphisms. It is the aim of the present article to extend that classification to locally finite $n$-cycles of linear mappings (see Section \ref{locallyfiniteSection} for a definition of such cycles).

Before we can proceed, we need to consider the case of a single endomorphism (i.e.\ the case $n=1$).

An endomorphism $u$ of a vector space $U$ is \textbf{locally finite} when, for every $x \in U$,
there exists a non-zero polynomial $p(t) \in \F[t]$ such that $p(u)[x]=0$ (or, equivalently, the family $(u^k(x))_{k \geq 0}$ is linearly dependent).
Moreover, $u$ is called \textbf{locally nilpotent} when, for every $x \in U$, there exists an integer $n \geq 0$ such that $u^n(x)=0$, or equivalently
$U=\underset{n \in \N}{\bigcup} \Ker u^n$. Finally, given $p \in \Irr(\F)$, we say that $u$ is \textbf{$p$-primary}
when $U=\underset{n \in \N}{\bigcup} \Ker p(u)^n$.

When $u$ is locally finite over a space with countable dimension, the similarity class of $u$ is controlled by
a family of invariants called the Kaplansky invariants of $u$: we recall their definition below.
For all $p \in \Irr(\F)$, we set
$$U^{(p)}:=\underset{k \geq 0}{\bigcup} \Ker p(u)^k.$$
Assume now that $u$ is locally finite. The kernel decomposition theorem then yields
$$U=\underset{p \in \Irr(\F)}{\bigoplus} U^{(p)}.$$
Moreover, for any $p \in \Irr(\F)$, the subspace $U^{(p)}$ is stable under $u$, and the resulting endomorphism $u_p$
is $p$-primary. We say that $u_p$ is the \textbf{$p$-primary} part of $u$.

It is then easily seen that two locally finite endomorphisms $u,v$ are similar if and only if, for every $p \in \Irr(\F)$,
their $p$-primary parts are similar. Hence, in order to classify the locally finite endomorphisms, it suffices to classify the ones that are $p$-primary, for each $p \in \Irr(\F)$.

Now, fix $p \in \Irr(\F)$, and let $u$ be an endomorphism of a vector space $U$.
We endow $U$ with a structure of $\F[t]$-module by putting $q(t).x:=q(u)[x]$ for all $x \in U$ and all $q(t) \in \F[t]$.
The isomorphism type of this module is governed by the similarity type of $u$.
A linear subspace $V$ of $U$ is called \textbf{globally invariant} under $p(u)$ when $p(u)(V)=V$.
The sum $\overline{U}$ of all such subspaces is itself globally invariant under $p(u)$, and it is an $\F[t]$-submodule of $U$.
The space $\Ker p(u) \cap \overline{U}$ is naturally endowed with a structure of $\F[t]/(p)$-vector space: we denote
its dimension by $\kappa_{\infty}(u,p)$ (seen as a cardinal).
It can then be shown that the isomorphism type of the module $\overline{U}$ depends only on that dimension.
Moreover, one shows that $\kappa_{\infty}(u,p)=\kappa_{\infty}(u_p,p)$.

Now, we define the other Kaplansky invariants.
We define a non-decreasing sequence of submodules of $U$ indexed over the ordinals as follows:
\begin{itemize}
\item $p^0\,U:=U$;
\item $p^\alpha\,U=p.(p^{\alpha-1}\,U)$ for every ordinal $\alpha$ with a predecessor;
\item $p^\alpha\,U=\underset{\beta<\alpha}{\bigcap} p^\beta\,U$ for every limit ordinal $\alpha$.
\end{itemize}
Then, for every ordinal $\alpha$, the quotient module $p^\alpha U/p^{\alpha+1}U$ has a natural structure of vector space over
the residue field $\F[t]/(p)$, the
mapping $x\in U \mapsto p.x \in U$ induces a surjective homomorphism
$$u_{\alpha,p} : p^\alpha U/p^{\alpha+1}U \rightarrow p^{\alpha+1}U/p^{\alpha+2}U,$$
of $\F[t]/(p)$-vector spaces, and the dimension over $\F[t]/(p)$ of the kernel of that homomorphism is denoted by $\kappa_\alpha(u,p)$
and called the \textbf{Kaplansky invariant} of order $\alpha$ of $u$ with respect to $p$.
Again, one shows that $\kappa_\alpha(u_p,p)=\kappa_\alpha(u,p)$.
If $U$ has countable dimension, one proves that $p^\alpha U=p^{\alpha+1} U$ for every uncountable ordinal $\alpha$,
in which case $\kappa_\alpha(u,p)=0$.

The following result is a consequence of a general theorem of Kaplansky and Mackey
\cite{Kaplansky} on the classification of countably generated modules over a complete discrete valuation ring:

\begin{theo}[Mackey-Kaplansky]\label{KaplanskyTheo}
Let $u,v$ be locally finite endomorphisms of countable-dimensional vector spaces. The following conditions are equivalent:
\begin{enumerate}[(i)]
\item $u$ and $v$ are similar;
\item For every $p \in \Irr(\F)$ and every $\alpha \in \omega_1 \cup \{\infty\}$,
one has $\kappa_\alpha(u,p)=\kappa_\alpha(v,p)$.
\end{enumerate}
\end{theo}

The aim of the present paper is to generalize Kaplansky and Mackey's theorem to locally finite $n$-cycles of linear mappings, as defined in the next section.

\subsection{Reformulation in terms of quiver representations}

It is convenient to reformulate the above problem into a problem of quiver representations.
Remember that a \textbf{quiver} is essentially an oriented graph with potentially multiple oriented edges from one vertex to another.
More formally, we define a quiver $\calQ$ as a pair $(V,E)$ in which $V$ is a finite set and $E=(A_{a,b})_{(a,b)\in V^2}$
is a family of finite sets indexed over $V^2$.
For each $(a,b)\in V^2$, the elements of $A_{a,b}$ are called the \textbf{arrows} of $\calQ$ from $a$ to $b$.
Here, the quiver we are concerned with is the cyclic quiver $\calC_n$ represented as follows:

\begin{center}
\mbox{
\psfrag{0}{$0$}
\psfrag{1}{$1$}
\psfrag{2}{$2$}
\psfrag{dots}{$\cdots$}
\psfrag{n-2}{$n-2$}
\psfrag{n-1}{$n-1$}
\includegraphics{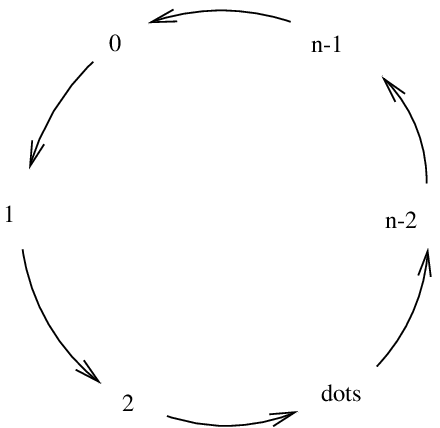}}
\end{center}

The set of vertices of $\calC_n$ is defined as $\Z/n$, and for each $(a,b)\in (\Z/n\Z)^2$, we put
$A_{a,b}:=\{a\}$ if $b=a+1$, otherwise $A_{a,b}:=\emptyset$.

Let $\calQ=(V,(A_{a,b})_{(a,b)\in V^2})$ be a quiver, and $\calC$ be a category.
A \textbf{representation of $\calQ$ on $\calC$} consists of a family $(U_a)_{a \in V}$ of objects of
$\calC$ and, for each $(a,b)\in V^2$, of a family $(f_x)_{x \in A_{a,b}}$ of morphisms of $\calC$ from $U_a$ to $U_b$.
In particular, a linear representation of $\calQ$ is a representation of $\calQ$ on the category of vector spaces over $\F$:
given such a linear representation $f=\bigl((U_a)_{a \in V},\bigl((f_x)_{x \in A_{a,b}}\bigr)_{(a,b)\in V^2}\bigr)$,
we define its dimension by $\dim f:=\sum_{a \in V} \dim U_a$.

Given two representations $f=\bigl((U_a)_{a \in V},\bigl((f_x)_{x \in A_{a,b}}\bigr)_{(a,b)\in V^2}\bigr)$ and \\
$g=\bigl((V_a)_{a \in V},\bigl((g_x)_{x \in A_{a,b}}\bigr)_{(a,b)\in V^2}\bigr)$ of $\calQ$ on a category $\calC$, a morphism $\Phi$ from $f$ to $g$
is a family $(\varphi_a)_{a \in V}$ such that, for every $a \in A$, we have $\varphi_a \in \Hom_\calC(U_a,V_a)$, and the following set of identities
is satisfied
$$\forall (a,b)\in V^2, \; \forall x \in A_{a,b}, \quad \varphi_b \circ f_x=g_x \circ \varphi_a.$$
$$\xymatrix{
U_a \ar[d]_{\varphi_a} \ar[r]^{f_x} & U_b \ar[d]^{\varphi_b} \\
V_a \ar[r]_{g_x}  & V_{b.}
}$$
Given three representations $f,g,h$ of the quiver $\calQ$ on $\calC$, and morphisms $\Phi=(\varphi_a)_{a \in V}$ and $\Psi=(\psi_a)_{a \in V}$, from $f$ to $g$ and from $g$ to $h$, respectively, the family
$\Psi \circ \Phi:=(\psi_a \circ \varphi_a)_{a \in V}$ is clearly a morphism from $f$ to $h$.
The above defines a category, denoted by $\Rep(\calQ,\calC)$, in which the objects are the representations of $\calQ$ on $\calC$,
and the morphisms and their composition have just been defined.

Next, given categories $\calC$ and $\calD$ and a functor $F : \calC \rightarrow \calD$, and
given a representation $f=\bigl((U_a)_{a \in V},\bigl((f_x)_{x \in A_{a,b}}\bigr)_{(a,b)\in V^2}\bigr)$ of $\calQ$ on $\calC$,
we define a representation $f^F :=((F(U_a))_{a \in V},\bigl((F(f_x))_{x \in A_{a,b}}\bigr)_{(a,b)\in V^2}\bigr)$
of $\calQ$ on $\calD$. Given a morphism $\varphi : f \rightarrow g$ of representations of $\calQ$ on $\calC$,
which we write $\varphi=(\varphi_a)_{a \in V}$, one checks that $\varphi^F:=(F(\varphi_a))_{a \in V}$ is a morphism
from $f^F$ to $g^F$ in the category $\Rep(\calQ,\calD)$.
Clearly, $F$ yields a functor from $\Rep(\calQ,\calC)$ to $\Rep(\calQ,\calD)$.

In order to proceed, we need the notion of a direct sum of linear representations of a quiver.
Let $(f_i)_{i \in I}$ be a family of linear representations of the quiver $\calQ$,
in which, for all $i \in I$, we write $f_i=\bigl((U_{a,i})_{a \in V},\bigl((f_{x,i})_{x \in A_{a,b}}\bigr)_{(a,b)\in V^2}\bigr)$.
The (external) direct sum of that family, denoted by $\underset{i \in I}{\bigoplus} f_i$,
is defined as the linear representation $\bigl((W_a)_{a \in V},\bigl((g_{x})_{x \in A_{a,b}}\bigr)_{(a,b)\in V^2}\bigr)$
in which $W_a=\underset{i \in I}{\bigoplus} U_{a,i}$ for all $a \in V$, and, for all $(a,b)\in V^2$ and all $x \in A_{a,b}$,
we define $g_x$ as the linear mapping from $W_a$ to $W_b$ that takes $\sum_{i \in I} z_i$ to
$\sum_{i \in I} f_{x,i}(z_i)$ for every family $(z_i)_{i \in I} \in \prod_{i \in I} U_{a,i}$ of vectors with finite support.
Note that the representation $\underset{i \in I}{\bigoplus} f_i$ is unmodified up to isomorphism if each
$f_i$ is replaced with an isomorphic representation.

\begin{Not}
A representation of the cyclic quiver $\calC_n$ on the category of sets simply consists of a family
$(X_k)_{k \in \Z/n}$ of sets and of a family $(f_k)_{k \in \Z/n}$ in which $f_k$ is a mapping from $X_k$ to
$X_{k+1}$ for all $k \in \Z/n$. In the remainder of the article, we will identify such a representation with the family
$(X_k,f_k)_{k \in \Z/n}$.
\end{Not}

A linear representation of the cyclic quiver $\calC_n$ simply consists of a family $(U_k)_{k \in \Z/n}$ of
vector spaces and of a list $(u_k)_{k \in \Z/n}$ in which $u_k$ is a linear mapping from $U_k$ to $U_{k+1}$ for all $k \in \Z/n\Z$.
Thus, linear representations of $\calC_n$ correspond to $n$-cycles of linear mappings, and one checks that
two linear representations of $\calC_n$ are isomorphic if and only if the corresponding $n$-cycles of linear mappings
are equivalent (in the meaning of Section \ref{problemSection}). Hence, our basic problem can be restated as follows:
\begin{center}
Classify, up to isomorphism, the linear representations of the cyclic quiver $\calC_n$.
\end{center}

Before we go on, let us introduce a family of examples of linear representations of
$\calC_n$. First of all, let us consider the ``shift" endomorphism $S$ of $\F[t]$ that takes $t^k$ to $t^{k-1}$ for all $k \geq 1$, and that takes
$t^0$ to $0$. Then, for all $i \in \lcro 1,n\rcro$, we define $U_{\overline{\imath}}:=t^{n-i} \F[t^n]$ and we note that
$S$ maps $U_k$ into $U_{k+1}$ for all $k \in \Z/n$.
The pair
$$j :=\bigl(U_k,S_{|U_k}\bigr)_{k \in \Z/n}$$
is then a linear representation of $\calC_n$ which we call the
\textbf{canonical infinite Jordan cycle-cell of base $0$.}
Note that $S_{|U_k}$ is injective unless $k=0$, in which case the kernel of $S_{|U_k}$ has dimension $1$.
Obviously, $j$ is locally nilpotent with countable dimension.

Next, given $l \in \Z/n$ and a linear representation $f=(U_k,u_k)_{k \in \Z/n}$ of $\calC_n$,
the family $R^l(f):=(U_{k-l},u_{k-l})_{k \in \Z/n}$ is a linear representation of $\calC_n$.
In particular $R^l(j)$ is a linear representation of $\calC_n$ called the \textbf{canonical infinite Jordan cycle-cell of base $l$.}
A linear representation $u=(U_k,u_k)$ of $\calC_n$
is called an \textbf{infinite Jordan cycle-cell of base $l$} whenever $u$ is isomorphic to $R^l(j)$, i.e.\ when
there is a family $(e_i)_{i \in \N}$ such that, for all $k \in \Z/n$, the
family $(e_i)_{i \in \N \; \text{with}\; k+i=l}$ is a basis of $U_k$, and
$u_k(e_i)=e_{i-1}$ for every pair $(i,k)\in \N \times \Z/n$ such that $k+i=l$, where we convene that $e_{-1}=0$.

\begin{Def}
Let $f =(X_k,f_k)_{k \in \Z/n} \in \Rep(\calC_n,\Set)$. Let $k \in \Z/n$ and $x \in X_k$.
By induction, we define a sequence $x^{(k,f)}=(x_i)_{i \geq 0} \in \prod_{i \in \N} X_{k+i}$ as follows:
$$x_0:=x\quad \text{and} \quad \forall i \geq 0, \; x_{i+1}=f_{k+i}(x_i).$$
It is called the \textbf{sequence associated with $x$ for the pair $(k,f)$}.
We say that this sequence is \textbf{ultimately cyclic} whenever there exists $l \geq 0$ such that
$x_{l+n}=x_l$ (in that case, we find by induction that $x_{i+n}=x_i$ for all $i \geq l$).
\end{Def}

\subsection{Reducing the locally finite case to the locally nilpotent case}\label{locallyfiniteSection}

Let $\bigl(U_k,u_k\bigr)_{k \in \Z/n}$ be an $n$-cycle of linear mappings.
Given $p \in \Irr(\F)$ and $k \in \Z/n$, we set
$$U_{k,p}:=\underset{i \geq 0}{\bigcup} \Ker p^i\bigl((\pi u)_k\bigr).$$
Let $k \in \Z/n$. Noting that
\begin{equation}\label{pseudocommutation}
u_k \circ (\pi u)_k=(\pi u)_{k+1} \circ u_k,
\end{equation}
we find that for every polynomial $q \in \F[t]$, the endomorphism $u_k$ maps $\Ker q((\pi u)_k)$ into
$\Ker q((\pi u)_{k+1})$. It follows that $u_k$ maps $U_{k,p}$ into $U_{k+1,p}$ for all $p \in \Irr(\F)$.
Now, assume that $p \neq t$. Then, $(\pi u)_k$ induces an automorphism of
$\Ker p^i\bigl((\pi u)_k\bigr)$ for all $i \in \N$, and hence an automorphism of $U_{k,p}$.
It follows that $u_k$ is injective on $U_{k,p}$ and that $u_{k-1}$ maps $U_{k-1,p}$ onto $U_{k,p}$.
Applying this to every $k$ shows that $u_k$ induces an isomorphism from $U_{k,p}$ to $U_{k+1,p}$
for every $p \in \Irr(\F) \setminus \{t\}$.

Next, we have the following basic results:

\begin{lemma}
Let $k$ and $l$ belong to $\Z/n$. Then:
\begin{itemize}
\item $(\pi u)_k$ is locally finite if and only if $(\pi u)_l$ is locally finite.
\item $(\pi u)_k$ is locally nilpotent if and only $(\pi u)_l$ is locally nilpotent.
\end{itemize}
\end{lemma}

\begin{proof}
For each statement it suffices to consider the case where $k=l+1$ and to prove the direct implication.
Assume then that $(\pi u)_k$ is locally finite, and let $x \in U_{l.}$
Then $q\bigl((\pi u)_k\bigr)[u_l(x)]=0$ for some $q \in \F[t] \setminus \{0\}$.
Using identity \eqref{pseudocommutation}, we obtain $u_l\Bigl[q\bigl((\pi u)_l\bigr)[x]\Bigr]=0$.
Composing by $u_{l+n-1}\circ \cdots \circ u_{l+1}$, this yields
$(tq)\bigl((\pi u)_l\bigr)[x]=0$, and we conclude that $(\pi u)_l$ is locally finite.

Likewise, if $(\pi u)_k$ is locally nilpotent then, in the above proof, we can take $q=t^i$ for some non-negative integer $i$,
in which case $tq=t^{i+1}$. Hence, $(\pi u)_l$ is locally nilpotent whenever $(\pi u)_k$ is locally nilpotent.
\end{proof}

This justifies the next definition:

\begin{Def}
Let $u=(U_k,u_k)_{k \in \Z/n}$ be a linear representation of $\calC_n$. We say that it is \textbf{locally finite} (respectively, \textbf{locally nilpotent}) whenever
$(\pi u)_0$ is locally finite (respectively, locally nilpotent).
\end{Def}

Let $u=(U_k,u_k)_{k \in \Z/n}$ be a locally nilpotent linear representation of $\calC_n$, and let
$k \in \Z/n$ and $x \in U_k$. Denote by $(x_i)_{i \in \N}$ the sequence associated with $x$ for the pair $(k,u)$.
There exists an integer $p \geq 0$ such that $\bigl((\pi u)_k\bigr)^p(x)=0$, and hence
$x_i=0$ for every integer $i \geq np$.

The \textbf{local nilindex} of $x$ with respect to $u$ and $k$, denoted by $\nu_{u,k}(x)$ (or more simply by $\nu_u(x)$ when
there is no possible confusion on the class $k$), is defined as the least
integer $i \geq 0$ such that $x_i=0$.
Note that if $x \neq 0$ then $\nu_{u,k+1}(u_k(x))=\nu_{u,k}(x)-1$.

\vskip 3mm
We have seen earlier that if $(\pi u)_0$ is locally nilpotent then so is $(\pi u)_1$.
In contrast, the fact that $(\pi u)_0$ is an isomorphism does not imply that $(\pi u)_1$ is an isomorphism.
For example, if $n=2$, $U_0=\{0\}$, $U_1=\F$, $u_0$ is the linear mapping from $\{0\}$ and $\F$, and
$u_1$ the linear mapping from $\F$ to $\{0\}$, then $(\pi u)_0$ is bijective but $(\pi u)_1$ is the zero endomorphism
of $\F$.

\begin{lemma}
Let $(U_k,u_k)_{k \in \Z/n}$ be a linear representation of $\calC_n$. The following conditions are equivalent:
\begin{enumerate}[(i)]
\item All the mappings $u_k$ are isomorphisms.
\item All the mappings $(\pi u)_k$ are automorphisms.
\end{enumerate}
\end{lemma}

\begin{proof}
The implication (i) $\Rightarrow$ (ii) is obvious. We obtain the converse implication by noting that, for all
$k \in \Z/n$, the bijectivity of $(\pi u)_k$ implies the injectivity of $u_k$ and the surjectivity of $u_{k-1}$.
\end{proof}

\begin{Def}
A linear representation $(U_k,u_k)_{k \in \Z/n}$ of $\calC_n$ is called \textbf{regular} when all the $u_k$'s are isomorphisms.
\end{Def}

Of course, every linear representation of $\calC_n$ that is equivalent to a regular one is regular itself. Moreover,
the classification of regular cycles is entirely reduced to the one of automorphisms:

\begin{lemma}
Let $u$ and $v$ be isomorphic linear representations of $\calC_n$.
Then $(\pi u)_0$ is similar to $(\pi v)_0$.
\end{lemma}

\begin{proof}
Let us write $u=(U_k,u_k)_{k \in \Z/n}$ and $v=(V_k,v_k)_{k \in \Z/n}$.

There exists a family $(\varphi_k)_{k \in \Z/n}$
in which $\varphi_k$ is an isomorphism from $U_k$ to $V_k$ for all $k \in \Z/n$, and
$\varphi_{k+1} \circ u_k=v_k \circ \varphi_k$ for all $k \in \Z/n$. Thus
\begin{align*}
\varphi_0 \circ (\pi u)_0 & =(\varphi_0 \circ u_{\overline{n-1}}) \circ u_{\overline{n-2}} \circ \cdots \circ u_{\overline{0}} \\
& =v_{\overline{n-1}} \circ \varphi_{\overline{n-1}} \circ u_{\overline{n-2}} \circ \cdots \circ u_{\overline{0}} \\
& =\cdots=(v_{\overline{n-1}} \circ \cdots \circ v_{\overline{0}}) \circ \varphi_0=(\pi v)_0 \circ \varphi_0.
\end{align*}
As $\varphi_0$ is an isomorphism, this shows that $(\pi u)_0$ is similar to $(\pi v)_0$.
\end{proof}

\begin{prop}\label{regularclassprop}
Two regular linear representations $u$ and $v$ of $\calC_n$
are isomorphic if and only if the automorphisms $(\pi u)_0$ and $(\pi v)_0$ are similar.
\end{prop}

\begin{proof}
Let us write $u=(U_k,u_k)_{k \in \Z/n}$ and $v=(V_k,v_k)_{k \in \Z/n}$.

We have just shown that $(\pi u)_0$ and $(\pi v)_0$ are similar whenever $u$ and $v$ are isomorphic.
Conversely, assume that there exists an isomorphism $h : U_0 \overset{\simeq}{\rightarrow} V_0$ such that $h \circ (\pi u)_0=(\pi v)_0 \circ h$. For all $k \in \lcro 0,n-1\rcro$, set
$$\varphi_{\overline{k}}:=v_{\overline{k-1}} \circ \cdots \circ v_{\overline{0}} \circ h \circ u_{\overline{0}}^{-1} \circ \cdots \circ u_{\overline{k-1}}^{-1},$$
which is an isomorphism from $U_k$ to $V_k$.
For all $k \in \lcro 0,n-2\rcro$, it is then straightforward that
$$v_{\overline{k}} \circ \varphi_{\overline{k}}=\varphi_{\overline{k+1}} \circ u_{\overline{k}},$$
whereas
\begin{align*}
v_{\overline{n-1}} \circ \varphi_{\overline{n-1}}
& =(\pi v)_0 \circ h \circ   u_{\overline{0}}^{-1} \circ \cdots \circ u_{\overline{n-2}}^{-1} \\
& =h \circ (\pi u)_0 \circ u_{\overline{0}}^{-1} \circ \cdots \circ u_{\overline{n-1}}^{-1} \circ u_{\overline{n-1}}   \\
& =h \circ u_{\overline{n-1}}=\varphi_{\overline{0}} \circ u_{\overline{n-1}}.
\end{align*}
Hence, the family $(\varphi_k)_{k \in \Z/n}$ defines an isomorphism from $u$ to $v$ in the category of
linear representations of $\calC_n$.
\end{proof}

For any automorphism $u$ of a vector space $U$, we define the linear representation $w=(w_k,U)_{k \in \Z/n}$
in which $w_0=u$ and $w_k=\id_U$ for all $k \in \Z/n \setminus \{0\}$.
Note that $w$ is regular and that $(\pi w)_0=u$, and in particular $w$ is locally finite if and only if $u$ is locally finite.
Hence, classifying regular (respectively, regular and locally finite) $n$-cycles of linear maps up to isomorphism amounts to classifying automorphisms (respectively, locally finite automorphisms) of vector spaces up to similarity.

Next, we demonstrate that the problem of classifying locally finite linear representations of $\calC_n$
up to isomorphism can be reduced to the one of classifying locally nilpotent linear representations.
Let $u=(U_k,u_k)_{k \in \Z/n}$ be a locally finite linear representation of $\calC_n$.
Remember that, for all $p \in \Irr(\F)$ and all $k \in \Z/n$,  the linear maps
$u_k$ maps $U_{k,p}$ into $U_{k+1,p}$, moreover the resulting linear map is even an isomorphism if $p \neq t$.
Set
$$U_{k,\nil}:=U_{k,t} \quad \text{and} \quad U_{k,\reg}:=\underset{p \in \Irr(\F) \setminus \{t\}}{\bigoplus} U_{k,p.}$$
We deduce that $u_k$ induces an isomorphism
$$u_{k,\reg} : U_{k,\reg} \overset{\simeq}{\longrightarrow} U_{k+1,\reg,}$$
and a linear mapping
$$u_{k,\nil} : U_{k,\nil} \longrightarrow U_{k+1,\nil.}$$
The linear representations
$$u_{\nil}:=(U_{k,\nil},u_{k,\nil})_{k \in \Z/n} \quad \text{and} \quad
u_{\reg}:=(U_{k,\reg},u_{k,\reg})_{k \in \Z/n}$$
of $\calC_n$ are locally nilpotent for the former, and locally finite and regular for the latter. They are called the \textbf{locally nilpotent part} and the \textbf{regular part of $u$.} The locally finite linear representation $u$ is isomorphic to the direct sum of its locally nilpotent and regular parts.

Finally, given an isomorphism $(\varphi_k)_{k \in \Z/n}$ from a locally finite linear representation
$u=(U_k,u_k)_{k \in \Z/n}$ of $\calC_n$ to a locally finite linear representation $v=(V_k,v_k)_{k \in \Z/n}$ of $\calC_n$, the mapping $\varphi_k$ induces an isomorphism from $U_{k,\nil}$ to $V_{k,\nil}$ and one from $U_{k,\reg}$ to $V_{k,\reg}$ for all $k \in \Z/n$, and from there we see that $u_{\reg}$ is isomorphic to $v_{\reg}$, and $u_{\nil}$ is isomorphic to $v_{\nil}$.
Let us conclude:

\begin{prop}
Two locally finite linear representations of $\calC_n$ are isomorphic if and only if
their regular parts are isomorphic and their locally nilpotent parts are isomorphic.
\end{prop}

Thanks to Proposition \ref{regularclassprop}, Kaplansky's theorem (Theorem \ref{KaplanskyTheo}) covers the classification of the regular locally finite
linear representations of $\calC_n$ with countable dimension. Hence, it only remains to classify
the locally nilpotent linear representations of $\calC_n$ with countable dimension up to isomorphism.
In the next section, we introduce the relevant invariants.

\subsection{The cyclic Kaplansky invariants}\label{invariantsSection}

We start by recalling the notion of height with respect to a continuous chain of subsets.
Let $X$ be a nonempty set. A \textbf{continuous chain} in $X$ is defined
as a non-increasing ``collection"\footnote{To avoid any logical pitfall, such a chain is defined as a functional relation $f(\alpha,Y)$ (in two variables $\alpha$ and $Y$) whose domain is the collection of all ordinals, and whose range is included in the power set of $X$ (so that, for every ordinal $\alpha$, the set
$X_\alpha$ is the unique $\beta$ such that $f(\alpha,\beta)$ holds true).
In practice, the chains we will create are obtained by transfinite inductive processes.}
$(X_\alpha)_\alpha$ of subsets of $X$ indexed over all ordinals, such that
$X_\alpha=\underset{\beta<\alpha}{\bigcap} X_\beta$ for every limit ordinal $\alpha$, and $X_0=X$.

Let $\calX=(X_\alpha)_\alpha$ be such a chain, and denote by $X_\infty$ its intersection.
For all $x \in X \setminus X_\infty$, there is a unique ordinal $\alpha$ such that $x \in X_\alpha \setminus X_{\alpha+1}$:
it is called the \textbf{height} of $x$ with respect to $\calX$, and denoted by
$h_\calX(x)$. We convene that $h_{\calX}(x):=\infty$ for all $x \in X_\infty$.
The least ordinal $\alpha$ such that $X_\alpha=X_\infty$ is called the
\textbf{length} of $\calX$ and denoted by $\ell(\calX)$: if it is a limit ordinal, then it is the supremum of the heights of the elements of $X$
with respect to $\calX$, whereas if it has a predecessor then $\ell(\calX)-1$ is the greatest height with respect to $\calX$.

Now, let $u=(U_k,u_k)_{k \in \Z/n}$ be a representation of $\calC_n$ on the category of sets.

By transfinite induction, we define, for each $k \in \Z/n$, a collection $(U_{k,\alpha})_{\alpha}$ of subsets of $U_k$
indexed over the ordinals, as follows:
\begin{itemize}
\item $U_{k,0}:=U_k$ for all $k \in \Z/n$;
\item For every ordinal $\alpha$ that has a predecessor, and for every $k \in \Z/n$, we put
$U_{k,\alpha}:=u_{k-1}(U_{k-1,\alpha-1})$;
\item For every limit ordinal $\alpha$ and every $k \in \Z/n$, we put
$U_{k,\alpha}:=\underset{\beta<\alpha}{\bigcap} U_{k,\beta}$.
\end{itemize}
By transfinite induction, we prove that the sequences $(U_{k,\alpha})_\alpha$ are all non-increasing.
Let indeed $\alpha$ be an ordinal such that the sequences $(U_{k,\beta})_{\beta<\alpha}$, for $k \in \Z/n$,
are all non-increasing.
\begin{itemize}
\item If $\alpha$ is a limit ordinal, then the very definition of $U_{k,\alpha}$ shows that $U_{k,\alpha} \subset U_{k,\beta}$ for all $\beta<\alpha$ and all $k \in \Z/n$.
\item Assume now that $\alpha$ has a predecessor.
Let $\beta<\alpha$ and $k \in \Z/n$. By induction, we get that
$$U_{k,\alpha}=u_{k-1}(U_{k-1,\alpha-1}) \subset u_{k-1}(U_{k-1,\beta}) =U_{k,\beta+1}.$$
In particular, if $\alpha-1$ has a predecessor, we apply this to $\beta:=\alpha-2$ and we deduce that
$U_{k,\alpha} \subset U_{k,\alpha-1}$. If $\alpha-1$ is a limit ordinal, we apply the above to all $\beta<\alpha-1$, and we deduce that
$$U_{k,\alpha} \subset \underset{\beta<\alpha-1}{\bigcap} U_{k,\beta+1}=U_{k,\alpha-1.}$$
\end{itemize}

Now, for all $k \in \Z/n$ and all $x \in X_k$, we denote by $h_{k,u}(x)$ the height of $x$ with respect to
$(U_{k,\alpha})_\alpha$, called the \textbf{$k$-th height of $x$ with respect to $u$.}
We shall understate the pair $(u,f)$ whenever there is no risk of confusion.
We also set
$$\ell(u):=\max_{k \in \Z/n} \ell\bigl((U_{k,\alpha})_\alpha\bigr),$$
called the \textbf{length} of $u$.
For all $k \in \Z/n$, we set
$$U_{k,\infty}:=U_{k,\ell(u)}=\underset{\alpha}{\bigcap}\, U_{k,\alpha.}$$

Now, let $u=(U_k,u_k)_{k \in \Z/n}$ be a \emph{linear} representation of $\calC_n$.
By composing $f$ with the forgetful functor from the category of $\F$-vector spaces to the category of sets,
we can use the above construction to define the length of $f$. By induction, one sees that, for all
$k \in \Z/n$, the sequence $(U_{k,\alpha})_\alpha$ is a chain of \emph{linear subspaces} of $U_k$.
Using the axiom of choice, one shows that the length of $u$ is bounded above by the
dimension of $u$; in particular if $u$ has countable dimension then the length of $u$ is a countable ordinal.

We are now ready to define the relevant invariants.

\begin{Def}
We say that $u$ is \textbf{reduced} when $U_{k,\infty}=\{0\}$ for all $k \in \Z/n$.
We say that $u$ is \textbf{saturated} when $U_{k,\infty}=U_k$ for all $k \in \Z/n$, i.e.\ when all the mappings $u_k$
are surjective.
\end{Def}

\begin{Def}
Let $k \in \Z/n$. The $(k,\infty)$-th \textbf{Kaplansky invariant} of $u$ is defined as the cardinal
$$\kappa_{k,\infty}(u):=\dim(U_{k,\infty} \cap \Ker u_k).$$
\end{Def}

Letting $k \in \Z/n$ and $\alpha$ be an ordinal, the mapping
$u_k$ induces a surjective linear mapping from $U_{k,\alpha}$ to $U_{k+1,\alpha+1}$ as well as a linear mapping from
$U_{k,\alpha+1}$ to $U_{k+1,\alpha+2}$, and hence it induces a surjective linear mapping
$$u_{k,\alpha} : U_{k,\alpha}/U_{k,\alpha+1} \twoheadrightarrow U_{k+1,\alpha+1}/U_{k+1,\alpha+2}.$$
We define the \textbf{$(k,\alpha)$-th Kaplansky invariant} of $u$ as
$$\kappa_{k,\alpha}(f):=\dim(\Ker u_{k,\alpha}).$$
We immediately give another interpretation of this Kaplansky invariant.

\begin{lemma}\label{lemma:reinterpretinvariant}
The invariant $\kappa_{k,\alpha}(f)$ is also the dimension of the quotient space $(\Ker u_k\cap U_{k,\alpha})/(\Ker u_k\cap U_{k,\alpha+1})$.
\end{lemma}

It follows that $\kappa_{k,\alpha}(f)>0$ if and only $\Ker u_k$ contains an element $x$ of $k$-th height $\alpha$ with respect to $f$.

\begin{proof}
We prove that the natural linear mapping $w$ from $E:=(\Ker u_k\cap U_{k,\alpha})/(\Ker u_k\cap U_{k,\alpha+1})$
to the kernel of $u_{k,\alpha}$ is an isomorphism. It is obviously injective, and hence we only need to prove its surjectivity.
So, let $x \in U_{k,\alpha}$ be such that $u_k(x) \in U_{k+1,\alpha+2}$. Hence $u_k(x)=u_k(x')$ for some $x' \in U_{k,\alpha+1}$,
to the effect that $x-x' \in (\Ker u_k) \cap U_{k,\alpha}$. Hence, the class of $x$ in $U_{k,\alpha}/U_{k,\alpha+1}$
equals the one of $x-x'$, which is the image of the class of $x-x'$ under $w$. Hence, $w$ is surjective, as claimed.
\end{proof}

All those invariants are additive with respect to direct sums:
let indeed $(u^{(i)})_{i \in I}$ be a family of linear representation of $\calC_n$.
For each $i \in I$, write $u^{(i)}=\bigl(U^{(i)}_{k},u^{(i)}_{k}\bigr)_{k \in \Z/n}$
and define $f=\bigl(U_k, u_k\bigr)_{k \in \Z/n}$ as the direct sum of the $f_i$'s.
By induction, one sees that for every $\alpha$ (ordinal or $\infty$) and every $k \in \Z/n$,
$$U_{k,\alpha}=\underset{i \in I}{\bigoplus}\, U^{(i)}_{k,\alpha},$$
and it easily follows that for every $\alpha$ (ordinal or $\infty$) and every $k \in \Z/n$,
$$\kappa_{k,\alpha}(u)=\sum_{i \in I} \kappa_{k,\alpha}(u^{(i)}).$$

Finally, let $u=(U_k,u_k)_{k \in \Z/n}$ and $v=(V_k,v_k)_{k \in \Z/n}$ be
linear representations of $\calC_n$, and let $\varphi=(\varphi_k)_{k \in \Z/n}$ be a morphism from $u$ to $v$.
By transfinite induction, we find that, for every $k \in \Z/n$ and every ordinal $\alpha$,
$\varphi_k$ maps $U_{k,\alpha}$ into $V_{k,\alpha}$. If $\varphi$ is an isomorphism, it follows that
$\varphi_k$ induces an isomorphism from $U_{k,\alpha}$ to $V_{k,\alpha}$, for every $k \in \Z/n$ and every $\alpha$ (ordinal or $\infty$). Then, for every $k \in \Z/n$ and every $\alpha$ (ordinal or $\infty$), the linear mapping $\varphi_k$ induces an isomorphism
$\varphi_{k,\alpha}$ from $U_{k,\alpha}/U_{k,\alpha+1}$ to $V_{k,\alpha}/V_{k,\alpha+1}$, and the following square
is obviously commutative:
$$\xymatrix{
U_{k,\alpha}/U_{k,\alpha+1} \ar[r]^<(0.15){u_{k,\alpha}} \ar[d]_{\varphi_{k,\alpha}}^{\simeq} & U_{k+1,\alpha+1}/U_{k+1,\alpha+2} \ar[d]^{\varphi_{k+1,\alpha+1}}_{\simeq} \\
V_{k,\alpha}/V_{k,\alpha+1} \ar[r]^<(0.15){v_{k,\alpha}} & V_{k+1,\alpha+1}/V_{k+1,\alpha+2.}
}$$
In particular, this yields
$$\kappa_{k,\alpha}(u)=\kappa_{k,\alpha}(v).$$

Moreover, assuming that $u$ and $v$ are isomorphic, $u$ is saturated if and only if $v$ is saturated, and
$u$ is reduced if and only if $v$ is reduced.

\subsection{Main results}

We are now ready to state our main results. Remember that $\omega_1$ denotes the first uncountable ordinal.
If a linear representation $u$ has countable dimension then its length is countable and it follows that
$\kappa_{k,\alpha}(u)=0$ for every uncountable ordinal $\alpha$ and every $k \in \Z/n$.

\begin{theo}[Cyclic Kaplansky Theorem]\label{mainTheoreduced}
Let $u$ and $v$ be reduced locally nilpotent linear representations of $\calC_n$ with countable dimension.
Then $u$ and $v$ are isomorphic if and only if
$\kappa_{k,\alpha}(u)=\kappa_{k,\alpha}(v)$ for all $k \in \Z/n$ and all $\alpha \in \omega_1$.
\end{theo}

\begin{theo}\label{mainTheogeneral}
Let $u$ and $v$ be locally nilpotent linear representations of $\calC_n$ with countable dimension.
Then $u$ and $v$ are isomorphic if and only if
$\kappa_{k,\alpha}(u)=\kappa_{k,\alpha}(v)$ for all $k \in \Z/n$ and all $\alpha \in \omega_1 \cup \{\infty\}$.
\end{theo}

To state our last main result, we need the notion of an adapted basis. Here, a basis of a vector space
is meant to be a \emph{subset} of it (not a family of vectors).

\begin{Def}
Let $u=(U_k,u_k)_{k \in \Z/n}$ be a linear representation of $\calC_n$.
An \textbf{adapted basis} for $u$ is a family $(\bfB_k)_{k \in \Z/n}$ in which, for all
$k \in \Z/n$, $\bfB_k$ is a basis of the vector space $U_k$ and $u_k$ maps every vector in $\bfB_k$ to a vector in $\bfB_{k+1}$ or to the zero vector of $U_{k+1}$.
\end{Def}

\begin{theo}[Adapted basis theorem]\label{adaptedbasisTheorem}
Let $u$ be a locally nilpotent linear representation of $\calC_n$ with countable dimension.
Then $u$ has an adapted basis.
\end{theo}

The remainder of the article is laid out as follows. In Section \ref{SaturatedSection}, we prove that
every locally nilpotent linear representation of $\calC_n$ splits into the direct sum of a saturated one and of
a reduced one. The classification of the saturated representations is carried out in Section \ref{SaturatedClassificationSection} and shown to be governed by the $\kappa_{k,\infty}(u)$ invariants.
The classification of the reduced representations is obtained in Section \ref{extensionSection} thanks to an adaptation of the
Mackey-Kaplansky extension theorem to cycles of linear mappings. In Section \ref{AdmissibleSection}, we study the possible
Kaplansky invariants for a reduced linear representation of $\calC_n$ with countable dimension, and from there we deduce Theorem \ref{adaptedbasisTheorem}.

In a subsequent article, we will use the results in the case $n=2$ to obtain a classification of the pairs
$(u,v)$ of locally finite endomorphisms of respective countable-dimensional vector spaces $U$ and $V$
for which there are linear maps $a \in \Hom(U,V)$ and $b \in \Hom(V,U)$ such that $b\circ a=u$ and $a\circ b=v$.

\section{Isolating and analyzing the saturated representations}\label{SaturatedSection}

\subsection{Subrepresentations, saturated and regular parts}

Let $u=(U_k,u_k)_{k \in \Z/n}$ be a linear representation of $\calC_n$.
We define a \textbf{subrepresentation} of $u$ as a family $V=(V_k)_{k \in \Z/n}$
in which:
\begin{itemize}
\item For all $k \in \Z/n$,  $V_k$ is a linear subspace of $U_k$;
\item For all $k \in \Z/n$, one has $u_k(V_k) \subset V_{k+1}$.
\end{itemize}
Such a subrepresentation then yields:
\begin{itemize}
\item The linear representation
$\bigl(V_k,(u_k)_{|V_k}\bigr)_{k \in \Z/n}$ of $\calC_n$,
denoted by $u_{|V}$;
\item The induced linear representation
$\bigl(U_k/V_k,\, \overline{u_k}\bigr)_{k \in \Z/n}$, where for each $k \in \Z/n$
one denotes by $\overline{u_k}$ the linear mapping from $U_k/V_k$ to $U_{k+1}/V_{k+1}$ induced by $u$.
We denote this representation by $u \modu V$.
\end{itemize}
If $u$ is locally nilpotent, then so are $u_{|V}$ and $u \modu V$.

For example, the family $U_\infty:=(U_{k,\infty})_{k \in \Z/n}$ is a subrepresentation of $u$,
and since $u_k(U_{k,\infty})=U_{k+1,\infty}$ for all $k \in \Z/n$, we see that
the representation $u_{|U_\infty}$ is saturated. We denote it by $u_\infty$ and call it the \textbf{saturated part of $u$.}
Besides, the induced representation $u \modu U_\infty$ is denoted by $u_{\red}$ and called the \textbf{reduced part of $u$}.
As we shall see shortly, $u_{\red}$ is reduced, which justifies the terminology.

Suppose now that we have two linear representations $u=(U_k,u_k)_{k \in \Z/n}$ and
$v=(V_k,v_k)_{k \in \Z/n}$ of $\calC_n$ and an isomorphism $\varphi$ from
$u$ to $v$. Let $(W_k)_{k \in \Z/n}$ be a subrepresentation of $u$. Then $W':=(\varphi_k(W_k))_{k \in \Z/n}$
is a subrepresentation of $v$, and one checks that $v_{|V'}$ is isomorphic to $u_{|V}$ and that
$v \modu V'$ is isomorphic to $u \modu V$.
In particular, we have seen towards the end of Section \ref{invariantsSection} that
$\varphi_k(U_{k,\infty})=V_{k,\infty}$ for all $k \in \Z/n$,
and hence $u_\infty$ is isomorphic to $v_\infty$, and $u_{\red}$ is isomorphic to $v_{\red}$.

We finish with two additional notions:
\begin{itemize}
\item Two subrepresentations $(V_k)_{k \in \Z/n}$ and  $(W_k)_{k \in \Z/n}$
are called \textbf{linearly disjoint} when $V_k \cap W_k=\{0\}$ for all $k \in \Z/n$.

\item We define an ordering of the set of all subrepresentations of $u$ as follows:
$(V_k)_{k \in \Z/n} \leq (W_k)_{k \in \Z/n}$ if and only if $V_k \subset W_k$ for all $k \in \Z/n$.
\end{itemize}

\subsection{Extracting a saturated cycle}\label{extractsaturatedSection}

\begin{prop}\label{direcfactorProp}
Let $u=(U_k,u_k)_{k \in \Z/n}$ be a locally nilpotent linear representation of $\calC_n$.
Then there exists a subrepresentation $(V_k)_{k \in \Z/n}$ of $u$ such that
$U_k=U_{k,\infty} \oplus V_k$ for all $k \in \Z/n$.
\end{prop}

\begin{proof}
Consider the set $\calX$ of all subrepresentations of $u$ that are linearly disjoint from $U_\infty$.
Let us show that this set is inductive. Firstly, it is non-empty (it contains $(\{0_{U_k}\})_{k \in \Z/n}$). Next,
consider a totally ordered subset $\calY$ of $\calX$, and, for each $k \in \Z/n$, define
$V_k:=\underset{Y \in \calY}{\bigcup} Y_k$, which is obviously a linear subspace of $U_k$ that is linearly disjoint from $U_{k,\infty}$,
and note that $u_k$ maps $V_k$ into $V_{k+1}$. Hence, $(V_k)_{k \in \Z/n}$ belongs to $\calX$, and it is obviously
greater than or equal to every element of $\calY$. Hence, by Zorn's lemma, there is
a maximal element $(V_k)_{k \in \Z/n}$ in $\calX$.

We claim that $U_k=U_{k,\infty} \oplus V_k$ for all $k \in \Z/n$.
Assume the contrary and choose $k \in \Z/n$ and $x \in U_k$ such that $x \not\in U_{k,\infty} \oplus V_k$.
Denote by $(x_i)_{i \geq 0}$ the sequence associated with $x$ for the pair $(k,u)$.
Since $u$ is locally nilpotent, we know that $x_i=0$ for some $i \geq 0$,
and hence we can find a minimal integer $i \geq 0$ such that $x_i \not\in U_{k+i,\infty} \oplus V_{k+i}$.
Replacing $k$ with $k+i$ and $x$ with $x_i$, we can simply assume that $u_k(x) \in U_{k+1,\infty} \oplus V_{k+1}$.
So, we write $u_k(x)=y+z$ where $y \in U_{k+1,\infty}$ and $z \in V_{k+1}$.
Since $U_{k+1,\infty}=u_k(U_{k,\infty})$, we have $y=u_k(y')$ for some $y' \in U_{k,\infty}$.
It follows that $x':=x-y'$ is outside of $U_{k,\infty} \oplus V_k$ and satisfies $u_k(x')=z \in V_{k+1}$.
Set then $V'_i:=V_i$ for all $i \in (\Z/n) \setminus \{k\}$, and $V'_k:=V_k \oplus \F x'$.
Then we see that $(V'_l)_{l \in \Z/n}$ is a subrepresentation of $u$ that is linearly disjoint from $(U_{k,\infty})_{k \in \Z/n}$, which contradicts the maximality of $(V_l)_{l \in \Z/n}$. Hence, $(V_l)_{l \in \Z/n}$ has the required properties.
\end{proof}

Next, choose $V=(V_k)_{k \in \Z/n}$ given by Proposition \ref{direcfactorProp},
and denote by $V_{k,\alpha}$ the subspaces associated with the linear representation $u_{|V}$.
By induction, one finds that
$$\forall \alpha, \; \forall k \in \Z/n, \; U_{k,\alpha}=U_{k,\infty} \oplus V_{k,\alpha}$$
and in particular $V_{k,\infty}=\{0\}$ for all $k \in \Z/n$. In other words, $u_{|V}$ is reduced.
It follows that
$$u \simeq u_\infty \oplus u_{|V} \quad \text{and} \quad u_{|V} \simeq u_{\red.}$$
In particular, $u_{\red}$ is reduced and
$$u \simeq u_\infty \oplus u_{\red}.$$

Hence, we can conclude:

\begin{theo}
Let $u$ and $v$ be locally nilpotent linear representations of $\calC_n$. Then the following conditions are equivalent:
\begin{enumerate}[(i)]
\item $u$ and $v$ are isomorphic;
\item $u_{\infty}$ is isomorphic to $v_\infty$, and $u_{\red}$ is isomorphic to $v_{\red}$.
\end{enumerate}
\end{theo}

Hence, it remains to classify the saturated locally nilpotent representations, and the reduced ones.
The former classification is carried out in the next section (with no restriction on the dimension), the latter
is carried out in Section \ref{extensionSection} for representations with countable dimension.

\subsection{The classification of saturated cycles}\label{SaturatedClassificationSection}

\begin{theo}\label{saturatedJordanTheo}
Let $u$ be a saturated locally nilpotent linear representation of $\calC_n$.
Then $u$ is isomorphic to a direct sum of infinite Jordan cycle-cells.
\end{theo}

\begin{proof}
Let us write $u=\bigl(U_k,u_k\bigr)_{k \in \Z/n}$.
For all $k \in \Z/n$, let us choose a basis $(e_{k,i})_{i \in I_k}$ of $\Ker u_k$ in which the indexing sets $I_k$ are pairwise disjoint. As each $u_l$ map is surjective, we find by induction, for all $k \in \Z/n$ and all $i \in I_k$,
a sequence $(e_{k,i,l})_{l \geq 0}$ of vectors in which:
\begin{itemize}
\item $e_{k,i,0}=e_{k,i}$ ;
\item For all $l \in \N$, $e_{k,i,l} \in U_{k-l}$ ;
\item For all $l >0$, $u_{k-l}(e_{k,i,l})=e_{k,i,l-1.}$
\end{itemize}
Fix $k \in \Z/n$ and denote by $J_k$ the set of all triples $(j,i,l')$ in which $j \in \Z/n\Z$, $i \in I_j$, $l'$ is a non-negative integer and $j-l'=k$;
for $l \geq 0$ and $k \in \Z/n$, we denote by $J_{k,l}$ the set of all such triples $(j,i,l')$ in which $l' \leq l$.
We claim that the family $\bfB_k:=(e_{j,i,l})_{(j,i,l)\in J_k}$ is a basis of $V_k$. First, we note that every vector in $\bfB_k$ is the image of some vector of $\bfB_{k-1}$ under $u_{k-1}$.

Given $l \in \N$, denote by $\bfB_{k,l}$ the subfamily of $\bfB_k$ obtained by
considering only the triples $(j,i,l')$ in $J_{k,l}$.
We shall prove by induction that $\bfB_{k,l}$ is linearly independent for all $l \geq 0$.
This is already true for $l=0$ because $\bfB_{k,0}$ equals $(e_{k,i})_{i \in I_k}$ up to a bijection of indexing sets.
Now, let $l \geq 1$ and assume that $\bfB_{k,l-1}$ is linearly independent for all $k \in \Z/n$.
Let $k \in \Z/n$, and let $(\lambda_{j,i,l'})_{(j,i,l')\in J_{k,l}}$ be a family of scalars with finite support such that
$$\sum_{(j,i,l')\in J_{k, l}} \lambda_{j,i,l'}\, e_{j,i,l'}=0.$$
Applying $u_{k+l-1} \circ \cdots \circ u_{k}$ to this identity leads to
$$\sum_{i \in I_{k}} \lambda_{k+l,i,l}\, e_{k+l,i}=0,$$
and hence $\lambda_{k+l,i,l}=0$ for all $i \in I_{k}$. Hence,
$$\sum_{(j,i,l')\in J_{k,l-1}} \lambda_{j,i,l'}\, e_{j,i,l'}=0,$$
which, by induction, leads to $\lambda_{j,i,l'}=0$ for all $(j,i,l') \in J_{k,l-1}$.
Hence, $\bfB_{k,l}$ is linearly independent for all $k \in \Z/n$ and all $l \geq 0$, and we conclude
that $\bfB_k$ is linearly independent for all $k \in \Z/n$.

Next, we prove by induction on $l$ that, for all $k \in \Z/n$, every $x \in U_k$ with local nilindex
$l$ is a linear combination of vectors of $\bfB_k$.
This is obvious for $l=0$. Assume that it holds for some integer $l$.
Let $k \in \Z/n$ and let $x \in U_k$ have local nilindex $l+1$.
Then $u_k(x)$ has local nilindex $l$, and hence
it is a linear combination of vectors of $\bfB_{k+1}$. Each vector of $\bfB_{k+1}$ is the image under $u_k$
of some vector of $\bfB_k$, whence there is a vector $x' \in \Vect(\bfB_k)$ such that $u_k(x)=u_k(x')$.
It follows that $x-x' \in \Ker u_k$, and hence $x-x' \in \Vect(e_{k,i})_{i \in I_k} \subset \Vect(\bfB_k)$.
We conclude that $x=(x-x')+x' \in \Vect(\bfB_k)$.

Therefore, $\bfB_k$ is a basis of $U_k$ for all $k \in \Z/n$.

Now, set $I:=\underset{k \in \Z/n}{\bigcup} I_k$.
Given $k \in \Z/n$ and $i \in I$, set $k' \in \Z/n$ such that $i \in I_{k'}$,
and denote by $V_{k,i}$ the linear span of the vectors $e_{k',i,l'}$ where $l'$ is such that $(k',i,l') \in J_{k}$.
We deduce from the above considerations that
$$\forall k \in \Z/n, \; V_{k}=\underset{i \in I}{\bigoplus}\, V_{k,i}.$$
Moreover, for all $i \in I$, the family $V^{(i)}:=(V_{k,i})_{k \in \Z/n}$ is a subrepresentation of $u$. It follows that
$$u \simeq \underset{i \in I}{\oplus} u_{|V^{(i)}.}$$
Finally, for all $k \in \Z/n$ and all $i \in I_k$, it is clear that $u_{|V^{(i)}}$
is an infinite Jordan cycle-cell with base $k$.
\end{proof}

Next, we prove that the $\kappa_{k,\infty}$ invariants determine the saturated locally nilpotent linear representations of
$\calC_n$ up to isomorphism.
Start from an $l \in \Z/n$ and an infinite Jordan cycle-cell $u$ of base $l$.
Clearly, $\kappa_{k,\infty}(u)=0$ for all $k \in (\Z/n) \setminus \{l\}$, and
$\kappa_{l,\infty}(u)=1$.

Now, let $f$ be a saturated locally nilpotent linear representation of $\calC_n$.
By Theorem \ref{saturatedJordanTheo}, there is a family $(u^{(i)})_{i \in I}$ of Jordan cycle-cells such that
$$f \simeq \underset{i \in I}{\oplus} u^{(i)}.$$
For all $k \in \Z/n$, denote by $I_k$ the set of all $i \in I$ such that $u^{(i)}$
has base $k$. The family $(I_k)_{k \in I}$ is a partition of $I$, and
obviously the data of the cardinality of the sets $I_k$ defines $f$ up to isomorphism.

Noting that the $\kappa_{k,\infty}$ invariants are additive with respect to direct sums and remembering that
$\kappa_{k,\infty}(u^{(i)})=1$ if $i \in I_k$, and $\kappa_{k,\infty}(u^{(i)})=0$ otherwise,
we find that $\kappa_{k,\infty}(f)=\card(I_k)$ for all $k \in \Z/n$.

Hence, we conclude:

\begin{theo}
Let $u$ and $v$ be saturated locally nilpotent linear representations of $\calC_n$.
Then $u$ is isomorphic to $v$ if and only if $\kappa_{k,\infty}(u)=\kappa_{k,\infty}(v)$ for all $k \in \Z/n$.
\end{theo}

Note that this result holds without any assumption on the respective dimensions of $u$ and $v$.

\vskip 3mm
Finally, the very definition of an infinite Jordan cycle-cell shows that it always has an adapted basis.
Combining this observation with Theorem \ref{saturatedJordanTheo}, we obtain:

\begin{theo}\label{saturatedadaptedbasisTheo}
Every saturated locally nilpotent linear representation of $\calC_n$ has an adapted basis.
\end{theo}

\section{The classification of reduced locally nilpotent cycles}\label{extensionSection}

The aim of the present section is to prove the Cyclic Kaplansky Theorem. This will be achieved thanks to a generalization of the celebrated extension theorem of Mackey and Kaplansky (see \cite{Fuchs,Kaplansky}).

\subsection{An extension theorem}\label{extensiontheoSubsection}

Let $u$ and $v$ be reduced linear representations of $\calC_n$, and $E \subset F$ be subrepresentations of $u$.
Let $\varphi : u_{|E} \rightarrow v$ and $\psi : u_{|F} \rightarrow v$ be morphisms.
We say that $\psi$ \textbf{extends} $\varphi$ whenever $(\psi_k)_{|E_k}=\varphi_k$ for all $k \in \Z/n$.
We say that $\varphi$ is \textbf{height-preserving} if
$$\forall k \in \Z/n, \; \forall x \in E_k, \; h_{k,v}\bigl(\varphi_k(x)\bigr)=h_{k,u}(x)$$
(beware here that $h_{k,u}(x)$ is meant to be the height of $x$ with respect to $u$, not with respect to $u_{|E}$!).

\begin{theo}[Extension Theorem]\label{extensionTHM}
Let $u$ and $v$ be reduced locally nilpotent linear representations of $\calC_n$
with countable dimension and the same cyclic Kaplansky invariants. Let $E$ be a finite-dimensional subrepresentation of $u$, and
$\varphi : u_{|E} \rightarrow v$ be a height-preserving morphism, i.e.\
$$\forall k \in \Z/n, \; \forall x \in E_k, \; h_{k,v}\bigl(\varphi_k(x)\bigr)=h_{k,u}(x).$$
Then $\varphi$ extends to an isomorphism from $u$ to $v$.
\end{theo}

The Cyclic Kaplansky Theorem is readily deduced by applying the Extension Theorem to the trivial morphism from the
zero subrepresentation of $u$ to $v$ (it preserves heights since the height of the zero vector is always $\infty$).

In order to simplify the proof of the extension theorem, it is convenient to restate it in terms of \emph{functional subrepresentations}
of $u \times v$. Here, $u=(U_k,u_k)_{k \in \Z/n}$ and $v=(V_k,v_k)_{k \in \Z/n}$ are reduced locally nilpotent linear representations of $\calC_n$.
We define the direct product representation
$$u \times v :=\bigl(U_k \times V_k ,w_k\bigr)_{k \in \Z/n}$$
where $w_k(x,y)=(u_k(x),v_k(y))$ for all $k \in \Z/n$ and all $(x,y)\in U_k \times V_k$ (note that $u \times v$ is isomorphic to $u \oplus v$,
but it is better to think in terms of cartesian products here).

Let $\Gamma=(\Gamma_k)_{k \in \Z/n}$ be a subrepresentation of $u \times v$.
For each (ordered) pair $(x,y)$ of sets, we set
$$\pi_1(x,y):=x, \quad \pi_2(x,y):=y \quad \text{and} \quad (x,y)^t:=(y,x)$$
and we extend $\pi_1$ and $\pi_2$ as functions on sets of pairs.
The first and second projections of $\Gamma$ are defined as
$$\pi_1(\Gamma):=\bigl(\pi_1(\Gamma_k)\bigr)_{k \in \Z/n} \quad \text{and} \quad \pi_2(\Gamma):=\bigl(\pi_2(\Gamma_k)\bigr)_{k \in \Z/n.}$$
They are subrepresentations of $u$ and $v$, respectively. Both are finite-dimensional if $\Gamma$ is finite-dimensional. The transpose of $\Gamma$ is defined as $\Gamma^t:=(\Gamma_k^t)_{k \in \Z/n}$, and one sees that it is a subrepresentation of
$v \times u$.

A subrepresentation $\Gamma$ of $u \times v$ is called \textbf{functional} whenever, for all $k \in \Z/n$ and all
$x \in U_k$, there is at most one $y \in V_k$ such that $(x,y)\in \Gamma_k$.
In that case, there is a uniquely-defined morphism $\varphi : u_{|\pi_1(\Gamma)} \rightarrow v$ such that
$$\forall k \in \Z/n, \; \Gamma_k=\{(x,\varphi_k(x)) \mid x \in \pi_1(\Gamma_k)\}.$$
Moreover, if both $\Gamma$ and $\Gamma^t$ are functional then the morphism $\varphi$ is actually an isomorphism from
$u_{|\pi_1(\Gamma)}$ to $v_{|\pi_2(\Gamma)}$.

Next, we say that a subrepresentation $\Gamma$ of $u \times v$ is \textbf{coherent} (with respect to $(u,v)$) when
$$\forall k \in \Z/n, \; \forall (x,y)\in \Gamma_k, \; h_{k,u}(x)=h_{k,v}(y).$$
In that case, we note that $\Gamma$ is functional: indeed, for all $k \in \Z/n$ and all $x \in U_k$,
given $y,y'$ in $V_k$ such that $(x,y)\in \Gamma_k$ and $(x,y')\in \Gamma_k$, we find that $(0_{U_k},y-y') \in \Gamma_k$ because $\Gamma_k$
is a linear subspace of $U_k \times V_k$; then $h_{k,v}(y-y')=h_{k,u}(0_{U_k})=\infty$ and hence $y-y'=0_{V_k}$
because $v$ is reduced. Likewise, as $\Gamma^t$ is obviously coherent with respect to $(v,u)$, we gather that it is also
functional. We conclude that if $\Gamma$ is coherent and $\pi_1(\Gamma)=(U_k)_{k \in \Z/n}$ and $\pi_2(\Gamma)=(V_k)_{k \in \Z/n}$, then
there is an isomorphism from $u$ to $v$.

Finally, let $E=(E_k)_{k \in \Z/n}$ be a subrepresentation of $u$ together with a height-preserving morphism $\varphi : u_{|E} \rightarrow v$.
Defining then $\Gamma_k$ as the graph of $\varphi_k$ for all $k \in \Z/n$, one checks that
$(\Gamma_k)_{k \in \Z/n}$ is a coherent subrepresentation of $u\times v$ with first projection $E$.

Hence, Theorem \ref{extensionTHM} can be conveniently reformulated as follows:

\begin{theo}[Graphic Extension Theorem]
Let $u=(U_k,u_k)_{k \in \Z/n}$ and $g=(V_k,v_k)_{k \in \Z/n}$ be reduced locally nilpotent linear representations of $\calC_n$
with countable dimension and the same cyclic Kaplansky invariants. Let $\Gamma$ be a
finite-dimensional coherent subrepresentation of $u \times v$. Then $\Gamma$ extends to a
coherent subrepresentation of $u \times v$ whose projections are $(U_k)_{k \in \Z/n}$ and $(V_k)_{k \in \Z/n}$.
\end{theo}

This theorem will be proved in three steps.

\begin{Def}
Let $u=(U_k,u_k)_{k \in \Z/n}$ be a reduced locally nilpotent linear represention of $\calC_n$.
Let $E$ be a subrepresentation of $u$, let $k \in \Z/n$, and let $x \in U_k$.
We say that $E$ \textbf{$k$-catches} $x$ when $x \in E_k$.
\end{Def}

\begin{theo}[Simple Extension Theorem]
Let $u=(U_k,u_k)_{k \in \Z/n}$ and $v$ be reduced locally nilpotent linear representations of $\calC_n$
that have the same cyclic Kaplansky invariants. Let $\Gamma$ be a
finite-dimensional coherent subrepresentation of $u \times v$, and let $k \in \Z/n$ and $x \in U_k$ be such that
$u_k(x) \in \pi_1(\Gamma_{k+1})$. Then there exists a finite-dimensional coherent subrepresentation $\Gamma' \geq \Gamma$ of $u \times v$
whose first projection $k$-catches $x$.
\end{theo}

\begin{theo}[Finite-Dimensional Extension Theorem]
Let $u=(U_k,u_k)_{k \in \Z/n}$ and $v$ be reduced locally nilpotent linear representations of $\calC_n$
that have the same cyclic Kaplansky invariants. Let $\Gamma$ be a
finite-dimensional coherent subrepresentation of $u \times v$, and let $k \in \Z/n$ and $x \in U_k$.
Then there exists a finite-dimensional coherent subrepresentation $\Gamma' \geq \Gamma$ of $u \times v$
whose first projection $k$-catches $x$.
\end{theo}

In the remainder of this section, we prove that the Simple Extension Theorem implies the
Finite-Dimensional Extension Theorem, and then we prove that the Finite-Dimensional Extension Theorem implies the Graphic Extension Theorem.

\begin{proof}[The Simple Extension Theorem implies the Finite-Dimensional Extension Theorem]
Assume that the Simple Extension Theorem holds.
We prove the Finite-Dimensional Extension Theorem by induction on the local nilindex of $x$. If that nilindex equals zero, then the result is obvious
because $x=0$ and hence we can simply take $\Gamma':=\Gamma$. Fix $N \in \N^*$ and assume that $x$ has local nilindex $N$.
Then $u_k(x)$ has local nilindex $N-1$, so by induction there is a finite-dimensional coherent subrepresentation $\Gamma' \geq \Gamma$
of $u \times v$ such that $u_k(x) \in \pi_1(\Gamma'_{k+1})$. By the Simple Extension Theorem, we obtain a finite-dimensional coherent subrepresentation $\Gamma'' \geq \Gamma'$ of $u \times v$ whose first projection $k$-catches $x$. Hence, $\Gamma'' \geq \Gamma$ and $\Gamma''$ has the required properties.
\end{proof}

\begin{proof}[The Finite-Dimensional Extension Theorem implies the Graphic Extension Theorem]
We use the traditional zigzag argument.
Assume that the Finite-Dimensional Extension Theorem holds. Since $u$ and $v$ have countable dimension,
we can find sequences $(e_i)_{i \geq 0}$ and $(f_i)_{i \geq 0}$ of vectors, together with mappings
$p : \N \rightarrow \Z/n$ and $q : \N \rightarrow \Z/n$ such that, for all $k \in \Z/n$,
$(e_i)_{i \in p^{-1}\{k\}}$ generates $U_k$ and $(f_i)_{i \in q^{-1}\{k\}}$ generates $V_k$.

By induction, we construct a non-decreasing sequence $(\Gamma^{(i)})_{i \geq 0}$ of finite-dimensional coherent subrepresentations of $u \times v$
as follows:
\begin{itemize}
\item We put $\Gamma^{(0)}:=\Gamma$;
\item Assume that, for some integer $N>0$, we have constructed a finite sequence $\Gamma^{(0)} \leq \cdots \leq \Gamma^{(N-1)}$
of finite-dimensional coherent subrepresentations of $u \times v$
such that, for all $i \in \lcro 0,N-1\rcro$, the first projection of $\Gamma^{(i)}$ $p(i)$-catches $e_i$, and
the second projection $q(i)$-catches $f_i$.
\end{itemize}
Applying the Finite-Dimensional Extension Theorem to the vector $e_N$ and to $\Gamma^{(N)}$, we obtain a finite-dimensional coherent
subrepresentation $\Delta \geq \Gamma^{(N)}$ of $u \times v$ whose first projection $p(N)$-catches $e_N$.
Then, applying the Finite-Dimensional Extension Theorem to the transposed subrepresentation $\Delta^t$ and the vector $f_N$, we find a finite-dimensional coherent
subrepresentation $\Delta' \geq \Delta^t$ of $v \times u$ whose first projection $q(N)$-catches $f_N$.
Hence, the transposed subrepresentation $\Gamma^{(N+1)}:=(\Delta')^t$ of $u \times v$ satisfies $\Gamma^{(N+1)} \geq \Delta$,
and its second projection $q(N)$-catches $f_N$. Obviously, its first projection also $p(N)$-catches $e_N$.

Hence, by induction, we have a full non-decreasing sequence $(\Gamma^{(i)})_{i \in \N}$ of finite-dimensional subrepresentations of
$u \times v$ such that $\Gamma^{(0)}=\Gamma$ and, for all $i \in \N$, the first projection of
$\Gamma^{(i)}$ $p(i)$-catches $e_i$ and the second one $q(i)$-catches $f_i$.
From there, set
$$\widetilde{\Gamma}:=\Bigl(\underset{i \in \N}{\bigcup}\, \Gamma_k^{(i)}\Bigr)_{k \in \Z.}$$
Using the fact that $(\Gamma^{(i)})_{i \in \N}$ is non-decreasing and that each $\Gamma^{(i)}$ is coherent, it is straightforward to check that $\widetilde{\Gamma}$ is a coherent subrepresentation of $u \times v$. Finally, for all $i \in \N$, the first projection of
$\widetilde{\Gamma}$ $p(i)$-catches $e_i$, and the second one $q(i)$-catches $f_i$.
It follows that the projections of $\widetilde{\Gamma}$ are exactly $(U_k)_{k \in \Z/n}$ and $(V_k)_{k \in \Z/n}$, which completes the proof of the
Graphic Extension Theorem.
\end{proof}

Now, all that remains is to prove the Simple Extension Theorem. In order to do this, we need a few basic results on heights:
they are given in the next two sections. The proof of the Simple Extension Theorem is given in Section \ref{proofextensionSection}.

\subsection{Basic results on heights}

In the present section and the next one, $u=(U_k,u_k)_{k \in \Z/n}$ denotes a reduced locally nilpotent linear representation of $\calC_n$.
Our first lemma is stated for representations of $\calC_n$ on the category of sets, as it will be reused in that general context later
in the article.

\begin{lemma}[Lifting principle]\label{lemmehauteurrelev}
Let $f=(X_k,f_k)_{k \in \Z/n}$ be a representation of $\calC_n$ on the category of sets.
Let $k \in \Z/n$ and let $y \in X_k$ be such that $h_{k,f}(y)$ is an ordinal that has a predecessor.
Then there exists $x \in X_{k-1}$ such that
$h_{k-1,f}(x)=h_{k,f}(y)-1$ and $y=f_{k-1}(x)$.
\end{lemma}

\begin{proof}
Set $\alpha:=h_{k,f}(y)$. Then $y \in X_{k,\alpha}$, and the very definition of $X_{k,\alpha}$
shows that $y=f_{k-1}(x)$ for some $x \in X_{k-1,\alpha-1}$. Moreover, the element $x$ cannot belong to
$X_{k-1,\alpha}$ otherwise $y$ would belong to $X_{k,\alpha+1}$. Hence, $h_{k-1,f}(x)=\alpha-1$.
\end{proof}

The next lemmas are special cases of standard results on the height with respect to continuous chains of linear subspaces
whose intersection is the zero subspace. We leave the proofs to the reader.

\begin{lemma}\label{lemmeliberte}
Let $k \in \Z/n$. Let $(x_i)_{i \in I}$ be a family of vectors of $U_k \setminus \{0\}$.
If $i \in I \mapsto h_{k,u}(x_i)$ is injective, then $(x_i)_{i \in I}$ is linearly independent.
\end{lemma}

\begin{lemma}\label{lemmehauteurelem}
Let $k \in \Z/n$, and let $x$ and $y$ belong to $U_k$.
Then $h_{k,u}(x+y) \geq \min(h_{k,u}(x),h_{k,u}(y))$, and equality holds if
$h_{k,u}(x) \neq h_{k,u}(y)$.
Besides, $h_{k,u}(\lambda x)=h_{k,u}(x)$ for all $\lambda \in \F \setminus \{0\}$.
\end{lemma}

From Lemma \ref{lemmeliberte}, one immediately deduces:

\begin{lemma}\label{finiteheightlemma}
Let $k \in \Z/n$, and let $E$ be a finite-dimensional linear subspace of $U_k$.
Then the set $h_{k,u}(E)$ is finite.
\end{lemma}

\begin{Def}
Let $k \in \Z/n$, and let $A_k$ be a linear subspace of $U_k$.
A vector $x\in U_k$ is called \textbf{$k$-adapted} to $A_k$ whenever
$$h_{k,u}(x)=\max \bigl(h_{k,u}(x+A_k)\bigr).$$
\end{Def}

As a corollary to Lemma \ref{finiteheightlemma} (applied to the linear subspace $\F x+A_k$), we have:

\begin{cor}\label{adaptedcor}
Let $k \in \Z/n$, $x \in U_k$, and $A_k$ be a linear subspace of $U_k$.
Then some vector of the affine subspace $x+A_k$ is $k$-adapted to $A_k$.
\end{cor}

Note that if $x\in A_k$, then $x$ is $k$-adapted to $A_k$ if and only if $x=0_{U_k}$.
Next, we give a useful characterization of $k$-adaptedness:

\begin{lemma}\label{caracadapteelem}
Let $k \in \Z/n$, and let $x \in U_k$ and $A_k$ be a linear subspace of $U_k$.
Then $x$ is $k$-adapted to $A_k$ if and only if
$$\forall y \in A_k, \; h_{k,u}(x+y)=\min\bigl(h_{k,u}(y),h_{k,u}(x)\bigr).$$
\end{lemma}

\begin{proof}
Assume that $x$ is not $k$-adapted to $A_k$.
Then there exists $y \in A_k$ such that $h_{k,u}(x)<h_{k,u}(x+y)$. Hence $h_{k,u}(x+y) \neq \min(h_{k,u}(x),h_{k,u}(y))$.

Assume now that $x$ is $k$-adapted to $A_k$.
Let $y \in A_k$. We already have
$\min(h_{k,u}(x),h_{k,u}(y)) \leq h_{k,u}(x+y) \leq h_{k,u}(x)$. If $h_{k,u}(x) \neq h_{k,u}(y)$,
the first inequality is an equality by Lemma \ref{lemmehauteurelem}. Otherwise the two inequalities yield $h_{k,u}(x+y)=h_{k,u}(x)=\min(h_{k,u}(x),h_{k,u}(y))$.
\end{proof}

\begin{lemma}\label{hauteuroutsideadapted}
Let $k \in \Z/n$,  $x \in U_k$, and $A_k$ be a linear subspace of $U_k$.
Assume that $h_{k,u}(x) \not\in h_{k,u}(A_k)$.
Then $x$ is $k$-adapted to $A_k$.
\end{lemma}

\begin{proof}
Let $y \in A_k$. As $h_{k,u}(x) \neq h_{k,u}(y)$, Lemma \ref{lemmehauteurelem} yields $h_{k,u}(x+y)=\min(h_{k,u}(x),h_{k,u}(y))$. The conclusion follows from Lemma \ref{caracadapteelem}.
\end{proof}

\begin{lemma}\label{adaptedimagelemma}
Let $k \in \Z/n$ and $x \in U_k$, and let $A_k$ be a linear subspace of $U_k$ that does not contain $x$.
Set $\alpha:=h_{k,u}(x)$.
Assume that $u_k(x)$ is $(k+1)$-adapted to $u_k(A_k \cap U_{k,\alpha})$
and that $h_{k+1,u}(u_k(x))=\alpha+1$. Then $x$ is $k$-adapted to $A_k$.
\end{lemma}

\begin{proof}
Let $y \in A_k$. Note that $x+y \neq 0_{U_k}$, whence $h_{k,u}(x+y)$ is an ordinal.
If $h_{k,u}(x+y)>h_{k,u}(x)$, then on the one hand $h_{k,u}(y)=h_{k,u}(x)$ and hence $u_k(y) \in u_k(A_k \cap U_{k,\alpha})$,
and on the other hand $h_{k+1,u}(u_k(x)+u_k(y))=h_{k+1,u}(u_k(x+y))> h_{k,u}(x+y) \geq \alpha+1=h_{k+1,u}(u_k(x))$;
but this would contradict the assumption that $u_k(x)$ is adapted to $u_k(A_k \cap U_{k,\alpha})$.
Hence, $h_{k,u}(x+y) \leq h_{k,u}(x)$. This shows that $x$ is $k$-adapted to $A_k$.
\end{proof}

\subsection{Cyclic Kaplansky invariants and adaptedness}

In order to prove the Simple Extension Theorem, it is necessary that we delve deeper into the meaning of the cyclic Kaplansky invariants.

Given $k \in \Z/n$ and a subrepresentation $A$ of $U$, it is convenient to put, for every ordinal $\alpha$,
$$A_{k,\alpha}:=\{x \in A_k : \; h_{k,u}(x) \geq \alpha\}=U_{k,\alpha} \cap A_k.$$
Beware however that this notation should not be confused with the notation that arises when one views $A$ as a linear representation of $\mathcal{C}_n$
(this would produce different subspaces in general).

Let $\alpha$ be an ordinal. Remember that the cardinal $\kappa_{k,\alpha}(u)$ is defined as the dimension of the kernel of the linear mapping
$$u_{k,\alpha} : U_{k,\alpha}/U_{k,\alpha+1} \rightarrow U_{k+1,\alpha+1}/U_{k+1,\alpha+2}$$
induced by $u_k$. This kernel is simply the quotient space
$\bigl[U_{k,\alpha} \cap u_k^{-1}(U_{k+1,\alpha+2})\bigr]/U_{k,\alpha+1}$.
Besides, one checks that
\begin{equation}\label{interpretinvariant}
U_{k,\alpha} \cap u_k^{-1}(U_{k+1,\alpha+2})=U_{k,\alpha+1}+(U_{k,\alpha} \cap \Ker u_k).
\end{equation}
Let indeed $x \in U_{k,\alpha} \cap u_k^{-1}(U_{k+1,\alpha+2})$. Then $h_{k+1,u}(u_k(x)) \geq \alpha+2$, and hence
$u_k(x)=u_k(y)$ for some $y \in U_{k,\alpha+1}$.
It follows that $u_k(x-y)=0$ and $x-y \in U_{k,\alpha}$,
from where we find that $x$ belongs to $U_{k,\alpha+1}+(U_{k,\alpha} \cap \Ker u_k)$.
The converse inclusion is straightforward.

Let us get back to the subrepresentation $A$.
We see that $A_{k,\alpha} \cap u_k^{-1}(A_{k+1,\alpha+2})$ is a linear subspace of $U_{k,\alpha} \cap u_k^{-1}(U_{k+1,\alpha+2})$
and that its intersection with $U_{k,\alpha+1}$ is obviously $A_{k,\alpha+1}$, yielding a natural linear injection
$$\bigl[A_{k,\alpha} \cap u_k^{-1}(A_{k+1,\alpha+2})\bigr]/A_{k,\alpha+1} \hookrightarrow
\bigl[U_{k,\alpha} \cap u_k^{-1}(U_{k+1,\alpha+2})\bigr]/U_{k,\alpha+1}=\Ker u_{k,\alpha}$$
whose range equals $\bigl[U_{k,\alpha+1}+\bigl(A_{k,\alpha} \cap u_k^{-1}(A_{k+1,\alpha+2})\bigr)\bigr]/U_{k,\alpha+1.}$

This allows us to characterize the adaptation of a vector of the kernel of $u_k$ to a subrepresentation of $u$:

\begin{lemma}\label{caracadaptedkernel}
Let $A$ be a finite-dimensional subrepresentation of $u$, and let $k\in \Z/n$ and $x \in U_k \setminus A_k$
be such that $u_k(x)=0$. Set $\alpha:=h_{k,u}(x)$.
Then $x$ is $k$-adapted to $A_k$ if and only if
$x \not\in U_{k,\alpha+1}+\bigl(A_{k,\alpha} \cap u_k^{-1}(A_{k+1,\alpha+2})\bigr)$.
\end{lemma}

\begin{proof}
Assume that $x$ is not $k$-adapted to $A_k$.
Then there exists $y \in A_k$ such that $h_{k,u}(x+y)>\alpha$.
We must have $h_{k,u}(y)=\alpha$ otherwise $h_{k,u}(x+y)\leq h_{k,u}(x)$ by Lemma \ref{lemmehauteurelem}.
Moreover, $u_k(y)=u_k(x+y) \in U_{k+1,\alpha+2}$ and $u_k(y) \in A_{k+1}$, and hence $-y \in A_{k,\alpha} \cap u_k^{-1}(A_{k+1,\alpha+2})$.
Writing $x=(x+y)+(-y)$, we conclude that $x$ belongs to $U_{k,\alpha+1}+(A_{k,\alpha} \cap u_k^{-1}(A_{k+1,\alpha+2}))$.

Conversely, assume that $x=y+z$ for some $y \in U_{k,\alpha+1}$ and some  $z \in A_{k,\alpha} \cap u_k^{-1}(A_{k+1,\alpha+2})$.
Then $x-z$ belongs to $x+A_k$ and $h_{k,u}(x-z)=h_{k,u}(y) \geq \alpha+1$, which shows that $x$ is not $k$-adapted to $A_k$.
\end{proof}

\subsection{Proof of the Simple Extension Theorem}\label{proofextensionSection}

Let $u$ and $v$ be reduced locally nilpotent linear representations of $\calC_n$ with the same cyclic Kaplansky invariants.
The proof of the Simple Extension Theorem is performed in three steps. In the first one, we prove the following basic result:

\begin{lemma}[Super-Elementary Extension Lemma]\label{superelemextensionlemma}
Let $\Gamma$ be a coherent subrepresentation of $u \times v$, let $k \in \Z/k$, and let
$(x,y) \in (U_k \times V_k) \setminus \Gamma_k$. Assume that $x$ is $k$-adapted to $\pi_1(\Gamma)_k$ and $y$ is $k$-adapted to $\pi_2(\Gamma)_k$.
Assume finally that $h_{k,u}(x)=h_{k,v}(y)$ and that $(u_k(x),v_k(y)) \in \Gamma_{k+1}$.
Define $\widetilde{\Gamma}=(\widetilde{\Gamma}_l)_{l \in \Z/n}$ as follows :
$\widetilde{\Gamma}_l:=\Gamma_l$ if $l \neq k$, and $\widetilde{\Gamma}_k:=\Gamma_k + \F(x,y)$.
We also write
$$\Gamma +_k (x,y):=\widetilde{\Gamma}.$$
Then $\widetilde{\Gamma}$ is a coherent subrepresentation of $u \times v$.
\end{lemma}

Note then that $\Gamma \leq \widetilde{\Gamma}$, that $\widetilde{\Gamma}$ is finite-dimensional if $\Gamma$
is finite-dimensional, and that the first projection of $\widetilde{\Gamma}$ $k$-catches $x$.

\begin{proof}
Using the assumption that $(u_k(x),v_k(y)) \in \Gamma_{k+1}$ and that $\Gamma$ is a subrepresentation of $u \times v$,
it is easily checked that $\widetilde{\Gamma}$ is a subrepresentation of $u \times v$.
Since $\Gamma$ is already coherent, it suffices to check that
$h_{k,u}(x')=h_{k,v}(y')$ for every pair $(x',y')\in \Gamma_k + \F (x,y)$.
However, for every $(z,z')\in \Gamma_k$, we deduce from Lemma \ref{caracadapteelem}
that
$$h_{k,u}(x+z)=\min\bigl(h_{k,u}(x),h_{k,u}(z)\bigr)=\min\bigl(h_{k,v}(y),h_{k,v}(z')\bigr)=h_{k,v}(y+z').$$
The conclusion then follows by applying the second point in Lemma \ref{lemmehauteurelem}
and by using the fact that $\Gamma$ is coherent.
\end{proof}

In the next step, we consider the problem of adding a vector of $U_k$ that belongs to the kernel of $u_k$.
This is the main part of the proof that requires the coincidence of the cyclic Kaplansky invariants:

\begin{lemma}[Second Super-Elementary Extension Lemma]\label{secondsuperelemextensionlemma}
Let $\Gamma$ be a coherent finite-dimensional subrepresentation of $u \times v$, let $k \in \Z/k$, and let
$x \in U_k$ be $k$-adapted to $\pi_1(\Gamma)_k$ and such that $u_k(x)=0$.
Then there exists a finite-dimensional coherent subrepresentation $\Gamma' \geq \Gamma$
whose first projection $k$-catches $x$.
\end{lemma}

\begin{proof}
Put $A:=\pi_1(\Gamma)$ and $B:=\pi_2(\Gamma)$. Those are finite-dimensional subrepresentations of $u$ and $v$, respectively.
If $x \in A_k$, we simply put $\Gamma':=\Gamma$. From now on, we assume that $x \not\in A_k$.
Set $\alpha:=h_{k,u}(x)$. Since $x$ is $k$-adapted to $A_k$,
Lemma \ref{caracadaptedkernel} shows that the class of $x$ modulo $U_{k,\alpha+1}$ is outside the range of the natural linear injection
$$\bigl[A_{k,\alpha} \cap u_k^{-1}(A_{k+1,\alpha+2})\bigr]/A_{k,\alpha+1} \hookrightarrow
\bigl[U_{k,\alpha} \cap u_k^{-1}(U_{k+1,\alpha+2})\bigr]/U_{k,\alpha+1}.$$
As the source space here is finite-dimensional, it follows that
$$\dim \bigl(\bigl[A_{k,\alpha} \cap u_k^{-1}(A_{k+1,\alpha+2})\bigr]/A_{k,\alpha+1}\bigr)< \kappa_{k,\alpha}(u).$$
The $k$-th component of the morphism from $u_{|A}$ to $v$ associated with $\Gamma$ preserves heights, and hence it induces an isomorphism from
$\bigl[A_{k,\alpha} \cap u_k^{-1}(A_{k+1,\alpha+2})\bigr]/A_{k,\alpha+1}$ to
$\bigl[B_{k,\alpha} \cap v_k^{-1}(B_{k+1,\alpha+2})\bigr]/B_{k,\alpha+1}$.
Since $\kappa_{k,\alpha}(v)=\kappa_{k,\alpha}(u)$, we deduce that the linear injection
$$\bigl[B_{k,\alpha} \cap v_k^{-1}(B_{k+1,\alpha+2})\bigr]/B_{k,\alpha+1} \hookrightarrow
\bigl[V_{k,\alpha} \cap v_k^{-1}(V_{k+1,\alpha+2})\bigr]/V_{k,\alpha+1}$$
is not surjective. In other words,
$$V_{k,\alpha+1}+(B_{k,\alpha} \cap v_k^{-1}(B_{k+1,\alpha+2})) \subsetneq V_{k,\alpha} \cap v_k^{-1}(V_{k+1,\alpha+2}).$$
As we have seen earlier (formula \eqref{interpretinvariant}) that $V_{k,\alpha} \cap v_k^{-1}(V_{k+1,\alpha+2})=V_{k,\alpha+1}+(V_{k,\alpha} \cap \Ker v_k)$,
this yields a vector $y$ that belongs to $V_{k,\alpha} \cap \Ker v_k$ but not to
$V_{k,\alpha+1}+(B_{k,\alpha} \cap v_k^{-1}(B_{k+1,\alpha+2}))$.
In particular, $v_k(y)=0$ and $y$ does not belong to $B_k$ (otherwise it would belong to
$B_{k,\alpha} \cap v_k^{-1}(B_{k+1,\alpha+2})$).
Finally, $h_{k,v}(y)=\alpha$ since $y$ belongs to $V_{k,\alpha} \setminus V_{k,\alpha+1}$.
Lemma \ref{caracadaptedkernel} then yields that $y$ is $k$-adapted to $B_k$.
The conclusion then follows directly from Lemma \ref{superelemextensionlemma}.
\end{proof}

We are now ready for the proof of the Simple Extension Theorem.
Let $u=(U_k,u_k)_{k \in \Z/n}$ and $v=(V_k,v_k)_{k \in \Z/n}$ be reduced locally nilpotent linear representations of $\calC_n$
with equal cyclic Kaplansky invariants. Let $\Gamma$ be a
finite-dimensional coherent subrepresentation of $u \times v$, and let $k \in \Z/n$ and $x \in U_k$ be such that
$u_k(x) \in \Gamma_{k+1}$.
Set $A:=\pi_1(\Gamma)$ and $B:=\pi_2(\Gamma)$.
Of course, we can replace $x$ with an arbitrary vector of the affine subspace $x+A_k$
that is $k$-adapted to $A_k$ (see Corollary \ref{adaptedcor}), and hence we lose no generality in assuming:
\begin{itemize}
\item[(H1)] The vector $x$ is $k$-adapted to $A_k$.
\end{itemize}
Indeed, if we have a finite-dimensional coherent subrepresentation $\Gamma' \geq \Gamma$ whose first projection $k$-catches
some vector of $x+A_k$, then $\pi_1(\Gamma')$ also $k$-catches $x$.

We assume (H1), and next we prove that no generality is lost in assuming further that:
\begin{itemize}
\item[(H2)] The vector $u_k(x)$ is $(k+1)$-adapted to $u_k(A_{k,h_{k,u}(x)})$.
\end{itemize}
Indeed, set $\beta:=h_{k,u}(x)$ and let $x' \in x+A_{k,\beta}$ be such that $h_{k+1,u}(u_k(x'))$ is the maximal height of the
vectors of $u_k(x)+u_k(A_{k,\beta})$. Obviously
$h_{k,u}(x') \geq \beta$, and hence $h_{k,u}(x')=\beta$ by (H1).
Hence, $x'$ satisfies both (H1) and (H2); once more, if the first projection of some
finite-dimensional coherent subrepresentation $\Gamma' \geq \Gamma$ $k$-catches $x'$
then it also $k$-catches $x$.

\vskip 3mm
From now on, we assume that both (H1) and (H2) hold.
Let us define $z$ as the only vector of $B_{k+1}$ such that
$$(u_k(x),z) \in \Gamma_{k+1}.$$
In the remainder of the proof, we set
$$\alpha:=h_{k+1,u}(u_k(x))=h_{k+1,v}(z).$$
If $\alpha=\infty$, then $u_k(x)=0_{U_{k+1}}$ and the conclusion follows directly from Lemma
\ref{secondsuperelemextensionlemma}.

In the remainder of the proof, we assume that $\alpha$ is an ordinal.
We split the discussion into two cases, whether $\alpha$ is a limit ordinal or not.
To this end, we note that $\alpha>h_{k,u}(x)\geq 0$.

\vskip 3mm
\noindent \textbf{Case 1: $\alpha$ is not a limit ordinal.} \\
Assume first that $h_{k,u}(x)=\alpha-1$.
As $\Gamma$ is coherent, we see from assumption (H2) that $z$ is $(k+1)$-adapted to $v_k(B_{k,\alpha-1})$,
whereas $h_{k+1,v}(z)=\alpha$.
By Lemma \ref{lemmehauteurrelev}, there exists a vector $y \in V_k$ such that $h_{k,v}(y)=\alpha-1$ and $v_k(y)=z$.
Noting that $z \not\in v_k(B_{k,\alpha-1})$, we gather that $y \not\in B_k$.
Lemma \ref{adaptedimagelemma} then shows that $y$ is $k$-adapted to $B_k$.
Then, Lemma \ref{superelemextensionlemma}  applies to the pair $(x,y)$ and yields the desired conclusion.

Assume now that $h_{k,u}(x)<\alpha-1$.
Again, we find a vector $x' \in U_k$ such that $u_k(x')=u_k(x)$ and $h_{k,u}(x')=\alpha-1$.
This time around, Lemma \ref{adaptedimagelemma} shows that $x'$ is $k$-adapted to $A_k$.
The first situation tackled in the above yields a vector $y' \in V_k$ such that
$\Gamma +_k (x',y')$ is a coherent subrepresentation of $u\times v$ whose first projection $k$-catches $x'$.
Next, we check that $x-x'$ is $k$-adapted to $\pi_1\bigl(\Gamma +_k (x',y')\bigr)_k=A_k+\F x'$:
indeed, for all $y \in A_k$ and all $\lambda \in \F$,
$$h_{k,u}(x+y) \leq h_{k,u}(x)<h_{k,u}(x') \leq h_{k,u}\bigl((\lambda-1)x'\bigr)$$
and hence, by Lemma \ref{lemmehauteurelem},
$$h_{k,u}(x-x'+y+\lambda x')=h_{k,u}\bigl(x+y+(\lambda-1)x'\bigr)=h_{k,u}(x+y) \leq h_{k,u}(x)=h_{k,u}(x-x').$$
Finally, $u_k(x-x')=0$, and hence Lemma \ref{secondsuperelemextensionlemma} yields a finite-dimensional coherent subrepresentation
$\Gamma' \geq \Gamma +_k (x',y') \geq \Gamma$ of $u \times v$ whose first projection $k$-catches $x-x'$. Therefore, $\pi_1(\Gamma')$ $k$-catches $x=(x-x')+x'$.
This completes the study of Case 1.

\vskip 3mm
\noindent \textbf{Case 2: $\alpha$ is a limit ordinal.} \\
Since $h_{k,u}(A_k+\F x)$ is finite (see Lemma \ref{finiteheightlemma}), we can find an ordinal
$\beta<\alpha$ for which $h_{k,u}(A_k+\F x)$ is disjoint from $\left)\beta,\alpha\right($.
We claim that there exists a vector $x' \in U_k$ such that $\beta<h_{k,u}(x')<\alpha$ and $u_k(x')=u_k(x)$.
To support this, we choose an ordinal $\gamma \in \left)\beta,\alpha\right($; as $u_k(x) \in U_{k+1,\alpha} \subset U_{k+1,\gamma+1}$, we have $u_k(x)=u_k(x')$ for some $x' \in U_{k,\gamma}$. Then, $h_{k,u}(x') \geq \gamma$, but it cannot be that
$h_{k,u}(x') \geq \alpha$ otherwise $h_{k+1,u}(u_k(x))>\alpha$.
Hence, $\beta<h_{k,u}(x')<\alpha$, as claimed.

Next, we set $\gamma:=h_{k,u}(x')$.
As $h_{k+1,u}(u_k(x'))>\gamma+1$, we readily have $\kappa_{k,\gamma}(u)>0$, and hence
$\kappa_{k,\gamma}(v)>0$. By Lemma \ref{lemma:reinterpretinvariant}, this yields a vector $z' \in \Ker v_k$
such that $h_{k,v}(z')=\gamma$.

Since $h_{k+1,v}(z)=h_{k+1,u}(u_k(x))=\alpha$, the same method as in the above yields a vector
$z'' \in V_k$ such that $\gamma<h_{k,v}(z'')<\alpha$ and $v_k(z'')=z$.
Now, we set $y':=z''+z' \in V_k$, so that $h_{k,v}(y')=\gamma$ and $v_k(y')=v_k(z'')=z$.

As $\gamma$ does not belong to $h_{k,u}(A_k)$, it does not belong to $h_{k,v}(B_k)$ either (because $\Gamma$ is coherent); since
$h_{k,u}(x')=\gamma=h_{k,v}(y')$, we deduce from Lemma \ref{hauteuroutsideadapted}
that $x'$ is $k$-adapted to $A_k$ and that $y'$ is $k$-adapted to $B_k$.
Hence, Lemma \ref{superelemextensionlemma} shows that $\Gamma':=\Gamma +_k (x',y')$ is a coherent finite-dimensional subrepresentation of $u \times v$
whose first projection $k$-catches $x'$.
Observing that $u_k(x-x')=0$, we shall prove that
$x-x'$ is $k$-adapted to $\pi_1(\Gamma'_k)$. Once this is done, it will suffice to
apply Lemma \ref{secondsuperelemextensionlemma} to the triple $(\Gamma',x-x',k)$, just like in Case 1.

Firstly, $h_{k,u}(x')\not\in h_{k,u}(A_k+\F x)$, and hence Lemma \ref{hauteuroutsideadapted} shows that
$x'$ is $k$-adapted to $A_k+\F x$. Let
$\lambda \in \F \setminus \{0\}$ and $a \in A_k$. We find
\begin{align*}
h_{k,u}(x+a+\lambda x') & =h_{k,u}(\lambda^{-1}(x+a)+x') \\
& =\min\bigl(h_{k,u}(\lambda^{-1}(x+a)),h_{k,u}(x')\bigr) \\
& =\min\bigl(h_{k,u}(x+a),h_{k,u}(x')\bigr).
\end{align*}
Besides, since $x$ is $k$-adapted to $A_k$,
$$\min\bigl(h_{k,u}(x+a),h_{k,u}(x')\bigr)=\min\bigl(h_{k,u}(x),h_{k,u}(a),h_{k,u}(x')\bigr).$$
Finally, since $x'$ is $k$-adapted to $A_k+\F x$, applying the same principle once more yields
$\min\bigl(h_{k,u}(a),h_{k,u}(\lambda x')\bigr)=h_{k,u}(a+\lambda x')$.
Therefore,
$$h_{k,u}(x+a+\lambda x')=\min\bigl(h_{k,u}(x),h_{k,u}(a+\lambda x')\bigr).$$
This still holds if $\lambda=0$ because $x$ is $k$-adapted to $A_k$. We conclude that $x$ is $k$-adapted to $A_k+\F x'$.

Hence,  $h_{k,u}(x)=\max h_{k,u}\bigl(x+(A_k+\F x')\bigr)$.
As $h_{k,u}(x')>h_{k,u}(x)$, we have $h_{k,u}(x-x')=h_{k,u}(x)$. Noting that
$(x-x')+(A_k+\F x')=x+(A_k+\F x')$, we conclude that $x-x'$ is $k$-adapted to $A_k+\F x'$, which completes the proof.

Hence, the Simple Extension Theorem is proved, and Theorem \ref{extensionTHM} follows as explained in Section \ref{extensiontheoSubsection}.

\section{The possible cyclic Kaplansky invariants}\label{AdmissibleSection}

The aim of this section is to determine the possible cyclic Kaplansky invariants for a regular locally nilpotent linear representation of $\calC_n$
with countable dimension. Remember that $\omega_1$ denotes the first uncountable ordinal, i.e.\ the set of all countable ordinals.

\subsection{Main results}

\begin{Def}
Let $m=(m_{k,\alpha})_{k,\alpha}$ be a family of cardinals indexed over $(\Z/n) \times \omega_1$.
We define its \textbf{support} as
$$\SUP(m):=\bigl\{(k,\alpha) \in (\Z/n) \times \omega_1 : m_{k,\alpha}>0\bigr\},$$
and its \textbf{combined support} as the second projection of its support, i.e.\
$$\CSup(m):=\bigl\{\alpha \in \omega_1 : \; \exists k \in \Z/n : m_{k,\alpha}>0\bigr\}.$$
\end{Def}

Given ordinals $\alpha<\beta$, remember that we denote by $\left)\alpha,\beta\right($ (respectively, by $\left(\alpha,\beta\right($)
the set of all ordinals $\gamma$ such that $\alpha<\gamma<\beta$ (respectively, $\alpha \leq \gamma<\beta$).

Our first observation is that the length of a linear representation of $\calC_n$ with countable dimension is countable:

\begin{lemma}
Let $u=(U_k,u_k)_{k \in \Z/n}$ be a linear representation of $\calC_n$ with countable dimension.
Then $\ell(u)$ is countable. Consequently, $\kappa_{k,\alpha}(u)=0$ for every uncountable ordinal $\alpha$ and ever $k\in \Z/n$.
\end{lemma}

\begin{proof}
Let $\alpha$ be an ordinal such that $U_{k,\alpha}=U_{k,\alpha+1}$ for all $k \in \Z/n$.
Then by transfinite induction one finds that $U_{k,\beta}=U_{k,\beta+1}$ for every ordinal $\beta \geq \alpha$ and every $k \in \Z/n$,
and one deduces that $\alpha \geq \ell(u)$.

Hence, for every ordinal $\alpha<\ell(u)$, we can choose an index $i_\alpha \in \Z/n$ and a vector $x_\alpha \in U_{i_\alpha,\alpha} \setminus
U_{i_\alpha,\alpha+1}$. For $k \in \Z/n$, set $L_k:=\{\alpha<\ell(u) : i_\alpha=k\}$. By Lemma \ref{lemmeliberte},
the family $(x_\alpha)_{\alpha \in L_k}$ is linearly independent in $U_k$ for all $k \in \Z/n$, and hence
$L_k$ is countable. Therefore $\underset{k \in \Z/n}{\bigcup} L_k$ is countable, and finally $\ell(u)$ is countable.
\end{proof}

The case $n=1$ is covered by the following known characterization, which is generally attributed to Zippin:

\begin{prop}
Let $(m_\alpha)_{\alpha \in \omega_1}$ be a family of countable cardinals.
The following conditions are equivalent:
\begin{enumerate}[(i)]
\item There exists a reduced locally nilpotent endomorphism $u$ of a countable-dimensional vector space such that
$\kappa_{\alpha}(u,t)=m_\alpha$ for all $\alpha \in \omega_1$.
\item The support of $(m_\alpha)_{\alpha \in \omega_1}$ is countable and, if we denote by $s$ its supremum, then for
every limit ordinal $\delta \leq s$ and every ordinal $\alpha<\delta$, there is an ordinal $\beta \in \left)\alpha,\delta\right($ such that
$m_\beta>0$.
\end{enumerate}
\end{prop}

In the general case, the characterization is both more surprising and more complicated.
It relies upon the following observation:

\begin{lemma}\label{supportconstraintlemma}
Let $u=(U_k,u_k)_{k \in \Z/n}$ be a linear representation of $\calC_n$.
Let $\alpha$ be a limit ordinal such that $\alpha<\ell(u)$.

Let $k \in \Z/n$. Assume that there is an ordinal $\beta<\alpha$ such that $\kappa_{k,\gamma}(u)=0$
for every ordinal $\gamma \in \left(\beta,\alpha\right($.
Then $\kappa_{k+l+1,\alpha+l}(u)=0$ for every integer $l \geq 0$.
\end{lemma}

\begin{proof}
To start with, we prove that $U_{k+1,\alpha}=U_{k+1,\alpha+1}$, i.e.\ $u_k(U_{k,\alpha})=U_{k+1,\alpha}$.
The inclusion $U_{k+1,\alpha+1} \subset U_{k+1,\alpha}$ is known, so we simply prove the converse one.

Let $y \in U_{k+1,\alpha}$.
Since $\alpha$ is a limit ordinal, we have $y \in U_{k+1,\beta+1}$ i.e.\ $y=u_k(x)$ for some
$x \in U_{k,\beta}$. Set $\gamma:=h_{k,u}(x)$. Then $\beta \leq \gamma$. If $\gamma<\alpha$ then,
as $h_{k+1,u}(y) \geq \alpha>\gamma+1$ we deduce that $\kappa_{k,\gamma}(u)>0$,
contradicting the assumptions. Hence, $\gamma \geq \alpha$ and we conclude that $y \in u_k(U_{k,\alpha})$.
Therefore, $U_{k+1,\alpha}=U_{k+1,\alpha+1}$.

From there, we obtain by induction that $U_{k+1+l,\alpha+l}=U_{k+1+l,\alpha+1+l}$
for every integer $l\geq 0$, and the conclusion follows immediately.
\end{proof}

\begin{Def}
Let $D$ be a subset of $(\Z/n) \times \omega_1$.
The \textbf{deficiency of $D$}, denoted by $\defi(D)$, is defined as the set of all pairs $(k,\delta) \in  (\Z/n) \times \omega_1$
in which $\delta$ is a limit ordinal and there exists
an ordinal $\alpha<\delta$
such that $\forall \beta \in \left)\alpha,\delta\right(, \; (k,\beta)\not\in D$.

We say that $D$ is \textbf{admissible} when it is countable and
$$\forall (k,\delta) \in \defi(D), \;
\forall l \in \N, \; (k+1+l,\delta+l) \not\in D.$$

A family of countable cardinals indexed over $(\Z/n) \times \omega_1$ is called \textbf{admissible}
whenever its support is admissible.
\end{Def}

\begin{Rem}
Let $\delta$ be a limit ordinal and let $k \in \Z/n$ be such that
$(k,\delta)\not\in \defi(D)$.
Let $\alpha \in D$ be such that $\alpha<\delta$. Then there is an ordinal $\beta \in \left)\alpha,\delta\right($ such that $(k,\beta)\in D$.
The set of all such ordinals $\beta$ is not finite, otherwise by taking the greatest such ordinal $\alpha'$ we would find no such ordinal in
$\left)\alpha',\delta\right($. Hence, there are infinitely many ordinals $\beta \in \left)\alpha,\delta\right($ such that $(k,\beta)\in D$.
\end{Rem}

Remember that the length of a linear representation of $\calC_n$ with countable dimension is countable.
Hence, using Lemma \ref{supportconstraintlemma}, we obtain the implication (i) $\Rightarrow$ (ii) in the following theorem, which gives a complete
characterization of the possible cyclic Kaplansky invariants.

\begin{theo}\label{admissibleTheo}
Let $(m_{k,\alpha})_{k \in \Z/n,\alpha\in \omega_1}$ be a family of countable cardinals.
The following conditions are equivalent:
\begin{enumerate}[(i)]
\item There exists a reduced locally nilpotent linear representation $u$ of $\calC_n$ with countable dimension such that
$\kappa_{k,\alpha}(u)=m_{k,\alpha}$ for all $(k,\alpha) \in (\Z/n) \times \omega_1$.
\item The family $(m_{k,\alpha})_{k \in \Z/n,\alpha\in \omega_1}$ is admissible.
\end{enumerate}
\end{theo}

The difficulty here is to prove that condition (ii) implies condition (i).

The proof of Theorem \ref{admissibleTheo} is spread over the next four sections.
The main key is a result called the Realization Theorem, which is stated in the end of Section \ref{terminalSection} and
proved over the course of the next three sections.
As an application of the Realization Theorem, we will prove the Adapted Basis Theorem in Section \ref{adaptedProofSection}.

\subsection{Terminal representations of $\calC_n$}\label{terminalSection}

Here, we consider the category $\PSet$ of pointed sets: its objects are the pairs
$(X,o)$ consisting of a set $X$ and of an element $o$ of $X$. A morphism from an object $(X,o)$ to an object $(Y,o')$
is a mapping $\varphi : X \rightarrow Y$ such that $\varphi(o)=o'$. The composition of morphisms is the one of the associated mappings.
We have a forgetful functor $\PSet \rightarrow \Set$ that takes the object $(X,o)$ to $X$, and the morphism
$\varphi : (X,o) \rightarrow (Y,o')$ to the underlying mapping $\varphi : X \rightarrow Y$.

Let $f \in \Rep(\calC_n,\PSet)$. We can view $f$ as a family $(X_k,o_k,f_k)_{k \in \Z/n}$.
Remember that given $k \in \Z/n$ and $x \in X_k$, we denote by $x^{(k,f)}$ the sequence associated with $x$ for the pair $(k,f)$.

We say that $f$ is a \textbf{terminal representation} of $\calC_n$ when it satisfies the following two conditions:
\begin{itemize}
\item[(i)] One has $X_{k,\infty}=\{o_k\}$ for all $k \in \Z/n$.
\item[(ii)] For all $k \in \Z/n$ and all $x \in X_k$, the sequence $x^{(k,f)}$ is ultimately $n$-periodic.
\end{itemize}

Assume that condition (i) holds. Then we immediately show that (ii) is equivalent to the following condition:
\begin{itemize}
\item[(iii)] For every $k \in \Z/n$ and all $x \in X_k$, one has $x^{(k,f)}_{ni}=o_k$ for some integer $i \geq 0$.
\end{itemize}
First, it is clear that (iii) implies (ii) because $f_k(o_k)=o_{k+1}$ for all $k \in \Z/n$.
Conversely, assume that (ii) holds. Let $k \in \Z/n$ and $x \in X_k$.
Then there is an index $i$ such that $x^{(k,f)}_{ni}=x^{(k,f)}_{n(i+1)}$.
Using the sequence $(x^{(k,f)}_{ni},x^{(k,f)}_{ni+1},\dots,x^{(k,f)}_{ni+n-1})$, we find by transfinite induction that
$x^{(k,f)}_{ni+l}$ belongs to $X_{k+l,\alpha}$ for every ordinal $\alpha$ and every integer $l \in \lcro 0,n-1\rcro$, and hence
condition (i) shows in particular that $x^{(k,f)}_{ni}=o_k$.

\vskip 3mm
Denote by $\F-\LinCat$ the category of vector spaces over $\F$.
We define a functor $\F-\Real : \PSet \rightarrow \F-\LinCat$ as follows:
for every pointed set $(X,o)$, its image under $\F-\Real$ is the vector space $\F^{(X \setminus \{o\})}$
of all families of scalars of $\F$ indexed over $X \setminus \{o\}$ and with finite support, we
denote by $(e_x)_{x \in X \setminus \{o\}}$ its standard basis (so that $(e_x)_y=\delta_{x,y}$ for all $(x,y) \in (X \setminus \{o\})^2$)
and we convene that $e_o$ is the zero vector of $\F^{(X \setminus \{o\})}$.
Then, for every morphism $f : (X,o) \rightarrow (Y,o')$ in $\PSet$, we define
its image under $\F-\Real$ as the sole linear map $\widetilde{f}$ from $\F^{(X \setminus \{o\})}$
to $\F^{(Y \setminus \{o'\})}$ such that $\widetilde{f}(e_x)=e_{f(x)}$ for all $x \in X \setminus \{o\}$.
Note that $\widetilde{f}(e_x)=e_{f(x)}$ for all $x \in X$.
One checks that the above correctly defines a functor $\F-\Real$ from the category of pointed sets to the one of $\F$-vector spaces.
Given a representation $f$ of $\calC_n$ on the category $\PSet$, we denote by $\widetilde{f}$ the linear representation of $\calC_n$
obtained by composing $f$ with the previous functor. We say that $\widetilde{f}$ is the \textbf{linear realization} of $f$
(over $\F$).

Now, let $f=(X_k,o_k,f_k)_{k \in \Z/n}$ be a terminal representation of $\calC_n$, and denote by
$\widetilde{f}=(U_k,\widetilde{f_k})_{k \in \Z/n}$ its linear realization.
By transfinite induction, one checks that
$$\forall \alpha, \;\forall k \in \Z/n, \; U_{k,\alpha}=\Vect(X_{k,\alpha})$$
and it follows that
$$\forall k \in \Z/n, \; U_{k,\infty}=\Vect(X_{k,\infty}).$$
Using condition (i), we deduce from the latter equality that $\widetilde{f}$ is reduced. Next, let $k \in \Z/n$. For all $x \in X_k$,
we know that $x^{(k,f)}_{ni}=o_k$ for some integer $i \geq 0$, and hence
$(e_x)^{(k,\widetilde{f})}_{ni}=e_{o_k}=0_{U_k}$. Since $(e_x)_{x \in X_k}$ generates $U_k$, it easily follows that
for all $y \in U_k$ there exists an integer $j \geq 0$ such that $y^{(k,\widetilde{f})}_{nj}=0_{U_k}$, and we
conclude that $\widetilde{f}$ is locally nilpotent.

Next, we show how the cyclic Kaplansky invariants of $\widetilde{f}$ can be derived from discrete invariants that are attached to $f$ itself.
For every ordinal $\alpha$ and every $k \in \Z/n$, we set
\begin{multline*}
n_{k,\alpha}(f):=\card\bigl\{x \in X_k : \; h_{k,f}(x)=\alpha \; \textrm{and} \; h_{k+1,f}(f_k(x)) >\alpha+1\bigr\} \\
+\sum_{y \in h_{k+1,f}^{-1}\{\alpha+1\}} \bigl(\card(f_k^{-1}\{y\} \cap h_{k,f}^{-1}\{\alpha\})-1\bigr),
\end{multline*}
where the sum is to be understood as a sum of potentially infinite cardinals, and where, for every
non-zero cardinal $\kappa$, we denote by $\kappa-1$ its predecessor if $\kappa$ is finite, otherwise $\kappa-1:=\kappa$
(so that, for any set $E$ with cardinality $\kappa$ and any element $x$ of $E$,
the set $E \setminus \{x\}$ has cardinality $\kappa-1$).
Note that $n_{k,\alpha}(f)=0$ for every ordinal $\alpha\geq \ell(f)$ and every $k \in \Z/n$.
Note also that if $X_{k,\alpha}$ is finite then
$$n_{k,\alpha}(f)=\card\bigl(h_{k,f}^{-1}\{\alpha\}\bigr)-\card\bigl(h_{k+1,f}^{-1}\{\alpha+1\}\bigr).$$
More generally, if we choose a subset $Z$ of $h_{k,f}^{-1}\{\alpha\}$ which $f_k$ maps bijectively onto
$h_{k+1,f}^{-1}\{\alpha+1\}$ (Lemma \ref{lemmehauteurrelev} proves the existence of $Z$), the following equality holds:
\begin{equation}\label{discretenumberseq}
n_{k,\alpha}(f)=\card\bigl(h_{k,f}^{-1}\{\alpha\}\setminus Z\bigr).
\end{equation}
Finally, note that if $\ell(f)$ has a predecessor, then
$n_{k,\ell(f)-1}(f)$ is the cardinality of the set of all $x \in X_k$ such that $h_{k,f}(x)=\ell(f)-1$.

We say that $n_{k,\alpha}(f)$ is the \textbf{discrete Kaplansky number} of $f$ with respect to the pair $(k,\alpha)$.
Those numbers are connected to the cyclic Kaplansky invariants of $\widetilde{f}$ as follows:

\begin{lemma}\label{realizationlemma}
For every $k\in \Z/n$ and every ordinal $\alpha$, one has
$$\kappa_{k,\alpha}(\widetilde{f})=n_{k,\alpha}(f).$$
\end{lemma}

\begin{proof}
Let $\alpha$ be an ordinal, and let $k \in \Z/n$.
For all $y \in X_{k+1}$ with $h_{k+1,f}(y)=\alpha+1$, we choose $z(y) \in X_k$ such that $f_k(z(y))=y$
and $h_{k,f}(z(y))=\alpha$ (see Lemma \ref{lemmehauteurrelev}).
Besides, a basis of the quotient space
$U_{k,\alpha}/U_{k,\alpha+1}$ is given by the classes of the vectors of type
$e_x$ where $x \in X_k$ satisfies $h_{k,f}(x)=\alpha$.
Denote by $V$ the linear subspace of $U_{k,\alpha}/U_{k,\alpha+1}$
spanned by the classes of vectors of the form $e_{z(y)}$.
Obviously, $f_k$ induces an isomorphism from $V$ to $U_{k+1,\alpha+1}/U_{k+1,\alpha+2}$,
and hence $\kappa_{k,\alpha}(\widetilde{f})$ equals the dimension of any subspace that is complementary to
$V$ in $U_{k,\alpha}/U_{k,\alpha+1}$.
One of those subspaces is the one, denoted by $W$, spanned by the classes of the vectors of the form
$e_x$, in which $x \in X_k$ satisfies $h_{k,f}(x)=\alpha$ and is distinct from all the $z(y)$ vectors.
Hence, by formula \eqref{discretenumberseq},
$$\kappa_{k,\alpha}(\widetilde{f})=\dim W=n_{k,\alpha}(f).$$
\end{proof}

We say that $f$ is \textbf{countable} when all the sets $X_k$, for $k \in \Z/n$, are countable.
Note that this is equivalent to having $\widetilde{f}$ of countable dimension.
The implication (ii) $\Rightarrow$ (i) in Theorem \ref{admissibleTheo} will thus come from the following result:

\begin{theo}[Realization Theorem]\label{realizationtheo}
Let $(m_{k,\alpha})_{(k,\alpha) \in (\Z/n) \times \omega_1}$ be an admissible family of cardinals. Then
there exists a countable terminal representation $f$ of $\calC_n$ such that
$n_{k,\alpha}(f)=m_{k,\alpha}$ for every $k \in \Z/n$ and every ordinal $\alpha\in \omega_1$.
\end{theo}

The next three sections are devoted to the proof of Theorem \ref{realizationtheo}.
First of all, we will explain two ways of constructing terminal representations from existing ones, and we will
understand how to compute the associated discrete Kaplansky numbers (Sections \ref{augmentationSection}
and \ref{pointedsumsSection}). The proof of Theorem \ref{realizationtheo} is carried out in Section \ref{proofrealizationtheoSection}.

\subsection{Augmentations of a terminal representation of $\calC_n$}\label{augmentationSection}

Let $f=(X_k,o_k,f_k)_{k \in \Z/n}$ be a terminal representation of $\calC_n$.
Denote by $D$ the support of $(n_{k,\alpha}(f))_{(k,\alpha) \in (\Z/n) \times \omega_1}$.

Given $l \in \Z/n$, we define the following property:

\begin{itemize}
\item[$(\Aprop_l)$] For every ordinal $\alpha <\ell(f)$, there exists
$x \in X_{l-1,\alpha} \setminus \{o_{l-1}\}$ such that $f_{l-1}(x)=o_l$.
\end{itemize}

\begin{Rem}\label{propAremark}
Property $(\Aprop_l)$ holds whenever $\ell(f)$ has a predecessor and $n_{l-1,\ell(f)-1}(f)>0$.
Assume indeed that $\ell(f)$ has a predecessor and that property $(\Aprop_l)$ fails. Then,
as $f_{l-1}(X_{l-1,\ell(f)-1})=X_{l,\ell(f)}=\{o_l\}$, we must have $X_{l-1,\ell(f)-1}=\{o_{l-1}\}$,
leading to $n_{l-1,\ell(f)-1}(f)=0$.
\end{Rem}

Assume now that property $(\Aprop_l)$ is satisfied.
Denote by $\gamma$ the least ordinal that does not belong to $X_l$ (this is a convenient way to choose an element outside of $X_l$).
Set $X_k^{+l}:=X_k$ and $o_k^{+l}:=o_k$ for all $k \in \Z/n \setminus \{l\}$, and
$X_l^{+l}:=X_l \cup \{\gamma\}$ and $o_l^{+l}:=\gamma$.

Set also $f_k^{+l}:=f_k$ for all $k \in \Z/n \setminus \{l-1,l\}$, and define
$f_{l-1}^{+l}$ and $f_l^{+l}$ as follows:
$$f_{l-1}^{+l}(o_{l-1})=\gamma \quad \text{and} \quad \forall y \in X_{l-1} \setminus \{o_{l-1}\}, \; f_{l-1}^{+l}(y)=f_{l-1}(y).$$
$$f_l^{+l}(\gamma)=o_{l+1}^+ \quad \text{and} \quad \forall y \in X_l, \; f_l^{+l}(y)=f_l(y)$$
(note that there is no confusion between these definitions if $n=1$).

Clearly,
$$f^{+l}:=(Y_k,o'_k,g_k)_{k \in \Z/n}:=(X_k^{+l},o_k^{+l},f_k^{+l})_{k \in \Z/n}$$
is a representation of $\calC_n$ on $\PSet$.
Moreover, for all $x \in X_l$, one has $x^{(l,f)}_{ni+(n-1)}=o_{l-1}$ for some minimal integer $i \geq 0$, and one easily deduces that
$x^{(l,f^{+l})}_{n(i+1)}=\gamma$. From there, one easily checks that $f^{+l}$ satisfies condition (ii) in the definition of a terminal representation of $\mathcal{C}_n$.
Since property $(\Aprop_l)$ is satisfied, we have in particular that
$X_{l-1,\alpha} \setminus \{o_{l-1}\} \neq \emptyset$ for every ordinal $\alpha<\ell(f)$.
From there, one shows by transfinite induction that, for every ordinal $\alpha\leq \ell(f)$,
$$\forall k \in \Z/n \setminus \{l\}, \;
Y_{k,\alpha}=X_{k,\alpha} \quad \text{and} \quad
Y_{l,\alpha}=X_{l,\alpha} \cup \{\gamma\}$$
In particular
$$\forall k \in \Z/n \setminus \{l\}, \;
Y_{k,\ell(f)}=X_{k,\ell(f)}=\{o_k\} \quad \text{and} \quad
Y_{l,\ell(f)}=\{o_l,\gamma\},$$
which leads to
$$\forall k \in \Z/n \setminus \{l\}, \;
Y_{k,\ell(f)+1}=\{o_k\} \quad \text{and} \quad Y_{l,\ell(f)+1}=\{\gamma\}.$$
Hence, $f^{+l}$ is a terminal representation of $\mathcal{C}_n$ with length $\ell(f^{+l})=\ell(f)+1$, and this length
is the one of the chain $(Y_{l,\alpha})_\alpha$.
We say that $f^{+l}$ is the \textbf{augmentation of $f$ at stage $l$}.

\begin{Rem}\label{augmentationremark}
Property $(\Aprop_{l+1})$ is obviously satisfied by $f^{+l}$ because $o_l \in Y_{l,\ell(f)} \setminus \{\gamma\}$
and $f^{+l}_l(o_l)=o_{l+1}^+$.
\end{Rem}

Next, we compute the discrete Kaplansky numbers of $f^{+l}$ from those of $f$.
The above identities of sets show that, for all $k \in \Z/n \setminus \{l\}$ and all
$x \in X_k$, one has
\begin{equation}\label{eq:augmentation1}
h_{k,f}(x)=h_{k,f^{+l}}(x).
\end{equation}
Moreover,
\begin{equation}\label{eq:augmentation2}
\forall x \in X_l \setminus \{o_l\}, \quad h_{l,f}(x)=h_{l,f^{+l}}(x) <\ell(f)
\quad \text{and} \quad
h_{l,f^{+l}}(o_l)=\ell(f).
\end{equation}
Using \eqref{eq:augmentation2}, we readily find that $n_{l,\ell(f)}(f^{+l})=1$, and
we also obtain $n_{k,\ell(f)}(f^{+l})=0$ for all $k \in \Z/n \setminus \{l\}$.

Combining the above identities also shows that
$n_{k,\alpha}(f)=n_{k,\alpha}(f^{+l})$ for every ordinal $\alpha<\ell(f)$ such that $\alpha+1 \neq \ell(f)$, and
every $k \in \Z/n$. Assume finally that $\ell(f)$ has a predecessor. Then, one checks that
$n_{k,\ell(f)-1}(f)=n_{k,\ell(f)-1}(f^{+l})$ for all $k \in \Z/n \setminus \{l-1\}$.
Finally, every $x \in X_{l-1}$ such that $h_{l-1,f}(x)=\ell(f)-1$ is mapped by $f_{l-1}^{+l}$ (and also by $f_{l-1}$) to $o_l$, which satisfies
$h_{l,f^{+l}}(o_l)=\ell(f)$. Hence,
$$n_{l-1,\ell(f)-1}(f^{+l})=\big|h_{l-1,f}^{-1}\{\ell(f)-1\}\big|-1=n_{l-1,\ell(f)-1}(f)-1.$$

Let us sum up:

\begin{prop}\label{augmentationprop}
Let $f$ be a terminal representation of $\calC_n$.
Let $l \in \Z/n$ satisfy $(\Aprop_l)$.

Then, $f^{+l}$ has length $\ell(f)+1$. Moreover,
$n_{k,\alpha}(f^{+l})=n_{k,\alpha}(f)$ for every $k \in \Z/n$ and every ordinal $\alpha<\ell(f)$, except for
$k=l-1$ and $\ell(f)=\alpha+1$, in which case $n_{k,\alpha}(f^{+l})=n_{k,\alpha}(f)-1$.
Finally, $n_{k,\ell(f)}(f^{+l})=\delta_{k,l}$ for all $k \in \Z/n$,
and $f^{+l}$ satisfies condition $(\Aprop_{l+1})$.
\end{prop}

Here is a simple application of the above construction. Fix $k \in \Z/n\Z$.
We start from the trivial terminal representation of $\calC_n$, which we denote by
$f^{(0)}:=(\{\emptyset\},\emptyset,\id_{\{\emptyset\}})_{l \in \Z/n}$.
Note that the length of $f^{(0)}$ is $0$ and that property $(\Aprop_l)$ is trivially satisfied for every $l \in \Z/n$
(because no ordinal is less than $0$!).
By induction, we define a sequence $(f^{(k,i)})_{i \geq 0}$ of terminal representations of
$\calC_n$, in which $f^{(k,i)}$ has length $i$ and satisfies condition $(\Aprop_{k+i})$, as follows:
$$f^{(k,0)}:=f^{(0)} \quad \text{and} \quad \forall i \geq 0, \; f^{(k,i+1)}=(f^{(k,i)})^{+(k+i)}.$$
Then, by induction we obtain that for every integer $i >0$,
$n_{l,\alpha}(f^{(k,i)})=0$ for every $l \in (\Z/n) \setminus \{k+i-1\}$ and every ordinal $\alpha$, whereas
$n_{k+i-1,\alpha}(f^{(k,i)})=\delta_{\alpha,i-1}$ for every ordinal $\alpha$.

From there, the Cyclic Kaplansky Theorem can be used to rediscover a known result on locally nilpotent linear representations of $\calC_n$ with finite
support (see \cite{Gelonch}):

\begin{theo}
Every locally nilpotent linear representation of $\calC_n$ with finite dimension is isomorphic to a direct sum of representations of the form
$\widetilde{f^{(k,i)}}$ with $k \in \Z/n$ and $i \geq 1$.
\end{theo}

\begin{proof}
Let $u$ be a linear representation of $\calC_n$ that is locally nilpotent and with finite dimension.
First, we note that $u$ is reduced. Indeed, if we write $u=(U_k,u_k)_{k \in \Z/n}$, then, for all $k \in \Z/n$,
 $(\pi u)_k$ induces an endomorphism of $U_{k,\infty}$ that is both nilpotent (because it is locally nilpotent and $U_{k,\infty}$
 is finite-dimensional) and surjective, and hence $U_{k,\infty}=\{0\}$ because $U_{k,\infty}$ is finite-dimensional.

Next, the length of $u$ is obviously finite, and all the cyclic Kaplansky invariants of $f$ are finite.
This shows that we can choose a finite set $I$ together with a mapping $t : I \rightarrow (\Z/n) \times \N^*$ in which, for all
$(k,\alpha)\in \Z/n \times \N$, the set $t^{-1}\{(k,\alpha+1)\}$ has cardinality $\kappa_{k+\alpha,\alpha}(u)$.
Now, set $g:=\underset{i \in I}{\bigoplus} \widetilde{f^{t(i)}}$, which is a (reduced) finite-dimensional locally nilpotent linear representation of $\calC_n$. By the above and Lemma \ref{realizationlemma}, we find that, for all $(k,\alpha)\in \Z/n \times \N$,
$$\kappa_{k,\alpha}(g)=\card t^{-1}\{(k-\alpha,\alpha+1)\}=\kappa_{k,\alpha}(u).$$
By the Cyclic Kaplansky Theorem, we conclude that $u$ is isomorphic to $g$.
\end{proof}

\subsection{Pointed sums of terminal representations}\label{pointedsumsSection}

Our next construction is the one of pointed sums.
Let $(f^{(i)})_{i \in I}$ be a family of terminal representations of $\calC_n$, indexed over a non-empty set $I$.
For all $i \in I$, let us write $f^{(i)}=\bigl(X_k^{(i)},o_k^{(i)},f_k^{(i)}\bigr)_{k \in \Z/n}$.
For all $k \in \Z/n$, we consider the external disjoint union
$$X_k=\coprod_{i \in I} X_k^{(i)}:=\bigl\{(i,x)\mid i \in I, \; x \in X_k^{(i)}\bigr\},$$
and we define the concatenation $f_k : X_k \rightarrow X_{k+1}$ of the mappings $f_k^{(i)}$ as follows:
$$\forall (i,x)\in \coprod_{i \in I} X_k^{(i)}, \quad f_k(i,x)=\bigl(i,f_k^{(i)}(x)\bigr).$$
For each $k \in \Z/n$, we define an equivalence relation
$\sim_k$ on $X_k$ (called \textbf{$k$-equivalence}) as follows:
$(i,x)\sim_k (j,y)$ if and only if $(i,x)=(j,y)$ or $(x,y)=(o_k^{(i)},o_k^{(j)})$.
Given $k \in \Z/n$, we denote by $X'_k$ the quotient set of $X_k$ under $\sim_k$, by $\pi_k : X_k \twoheadrightarrow X'_k$
the canonical projection, and by $o'_k$ the image of any $(i,o_k^{(i)})$ under $\pi_k$.
We note that, for all $k \in \Z/n$, the mapping $f_k$ takes any two $k$-equivalent elements of $X_k$ to two $(k+1)$-equivalent elements of $X_{k+1}$,
and hence it induces a mapping $f'_k : X'_k \rightarrow X'_{k+1}$; obviously $f'_k(o'_k)=o'_{k+1}$.
By transfinite induction, the following equalities are obtained:
$$\forall \alpha, \;\forall k \in \Z/n, \; X_{k,\alpha}=\coprod_{i \in I} X_{k,\alpha.}^{(i)}$$
In particular, for every ordinal $\alpha$ and every $k \in \Z/n$, one sees that
$X_{k,\alpha+1}=X_{k,\alpha}$ if and only if
$\forall i \in I, \;   X_{k,\alpha}^{(i)}=X_{k+1,\alpha}^{(i)}$, and from there
$\ell(f)$ is shown to be the supremum of the set $\{\ell(f_i) \mid i \in I\}$.
Moreover, for all $k \in \Z/n$, $i \in I$ and $x \in X_i$, the height of $(i,x)$ with respect to
the chain $(X_{k,\alpha})_{\alpha}$ equals the height of $x$ with respect to the chain $(X_{k,\alpha}^{(i)})_{\alpha}$.

By transfinite induction, one finds
$$\forall \alpha, \; \forall k \in \Z/n, \;  X'_{k,\alpha}=\pi_k(X_{k,\alpha}).$$
In particular,
$$\forall k \in \Z/n, \; \forall \gamma \geq \ell(f), \;  X'_{k,\gamma}=\pi_k\bigl(\{(i,o_k^{(i)}) \mid i \in I\}\bigr)=\{o'_k\}.$$
Finally, for every $k \in \Z/n$, every $i \in I$, every $x \in X_k^{(i)}$ and every integer $p \geq 0$, we find by induction that
$$(\pi_k(i,x))^{(k,f')}_p=\pi_k \bigl(i,x^{(k,f)}_p\bigr),$$
and hence $(\pi_k(i,x))^{(k,f')}_{nl}=o'_k$ for some integer $l \geq 0$.

We deduce that $(X'_k,o'_k,f'_k)$ is a terminal representation of $\calC_n$ whose length is at most $\ell(f)$.
Let $\alpha<\ell(f)$ be an ordinal. There exists $i \in I$ such that $\alpha<\ell(f_i) \leq \ell(f)$,
and hence we can find some $k \in \Z/n$ and some $x \in X_{k,\alpha}^{(i)} \setminus \{o_k^{(i)}\}$.
Set $y:=\pi_k(i,x)$. Then $y \in X'_{k,\alpha} \setminus \{o'_k\}$, and hence $\ell(f') >\alpha$.
It follows that $\ell(f') \geq \ell(f)$, and we conclude that $\ell(f')=\ell(f)=\sup\bigl\{ \ell(f_i)\mid i \in I\bigr\}$.

The terminal representation
$$\bigvee_{i \in I} f^{(i)}:=(X'_k,o'_k,f'_k)_{k \in \Z/n}$$
is called the \textbf{pointed sum} of the $f_i$ representations.
It is easily checked that the linear realization
of $\underset{i \in I}{\bigvee} f^{(i)}$ is isomorphic to the direct sum $\underset{i \in I}{\bigoplus} \widetilde{f^{(i)}}$
of the linear realizations of the terminal representations $f^{(i)}$.
Since the cyclic Kaplansky invariants are additive with respect to direct sums, we obtain
the cyclic Kaplansky numbers of $\underset{i \in I}{\bigvee} f^{(i)}$ as functions of those of the $f^{(i)}$'s, thanks to Lemma \ref{realizationlemma}:

\begin{prop}\label{pointedsumprop}
Let $(f^{(i)})_{i \in I}$ be a family of terminal representations of $\calC_n$ indexed over a non-empty set $I$.
Set $f':=\underset{i \in I}{\bigvee} f^{(i)}$.
One has
$$\ell(f')=\sup_{i \in I}\, \ell(f_i).$$
Moreover, for every ordinal $\alpha$ and every $k \in \Z/n$,
$$n_{k,\alpha}(f')=\sum_{i \in I} n_{k,\alpha}(f^{(i)}).$$
\end{prop}

\subsection{Proof of the Realization Theorem}\label{proofrealizationtheoSection}

\begin{Not}
Let $D$ be a countable subset of $(\Z/n) \times \omega_1$.

We define $\CSup(D)$ (the combined support of $D$) as the projection of $D$ on $\omega_1$, and the \textbf{supremum} of
$D$, denoted by $\Ubd(D)$, as the upper-bound of $\CSup(D)$.
\end{Not}

In particular, $\Ubd(\emptyset)=0$.

\begin{Not}
We denote by
$\calD$ the set of all non-empty admissible subsets of $(\Z/n) \times \omega_1$ for which the second projection has no greatest element, and
by $\calD'$ the set of all non-empty admissible subsets of $(\Z/n) \times \omega_1$.
\end{Not}

\begin{lemma}[First Partitioning Lemma]\label{splitlemma1}
There exists a mapping
$\varphi : \mathcal{D} \rightarrow \mathcal{D}^\N$ such that, for all $D \in \mathcal{D}$, the family $(\varphi_p(D))_{p \in \N}$ is a partition of $D$,
and for every $p \in \N$, $\varphi_p(D)$ has the same deficiency and the same supremum as $D$.
\end{lemma}

\begin{Def}
Let $D \in \mathcal{D}$. Set $\delta:=\Ubd(D)$.
A \textbf{good partition} of $D$ is a partition $(D_p)_{p \in \N} \in (\calD' \setminus \calD)^\N$ of $D$ in which, for every $k \in \Z/n$ such that $(k,\delta)\not\in \defi(D)$, and every ordinal $\alpha<\delta$, there exists $p \in \N$ and an ordinal $\beta \in \left)\alpha,\delta\right($ such that $(k,\beta) \in D_p$ and $\Ubd(D_p)=\beta$.
\end{Def}

\begin{lemma}[Second Partitioning Lemma]\label{splitlemma2}
Every element of $\mathcal{D}$ has a good partition.
\end{lemma}

The following proofs use a standard feature of ordinals:
every ordinal $\alpha$ splits uniquely into the sum
$\beta+k$ for some ordinal $\beta$ with no predecessor, and some non-negative integer $k$.

\begin{proof}[Proof of the First Partitioning Lemma]
Let $k \in \Z/n$.
Then $D_k:=\{\alpha \in \omega_1 : (k,\alpha)\in D\}$ is a subset of $\omega_1$ and hence it is well-ordered by $\leq$.
Hence, there is a unique increasing injection $f_k : D_k \rightarrow \omega_1$ such that $f_k(D_k)$ is an ordinal.
Next, we construct a mapping $G : \omega_1 \rightarrow \N$ as follows:
we choose a bijection $H : \N \overset{\simeq}{\rightarrow} \N^2$,
we denote by $L$ the set of all ordinals $\alpha \in \omega_1$ with no predecessor, and for every
such ordinal $\alpha$ and every integer $l \geq 0$, we put $G(\alpha+l)=\pi_2(H(l))$, where
$\pi_2 : (x,y)\in \N^2 \mapsto y$.
It follows that for all ordinals $\alpha<\beta$ in $\omega_1$ and all $p \in \N$,
if $\left)\alpha,\beta\right($ is infinite then it contains infinitely many elements that are mapped to $p$ by $G$.
Then, we set $F_k :=G \circ f_k : D_k \rightarrow \N$. Hence, for
all $\alpha<\beta$ in $D_k$ and all $p \in \N$, if there are infinitely many elements in
$\left)\alpha,\beta\right(\cap D_k$, then there are also infinitely many such elements that are mapped to $p$ by $F_k$.

Now, for all $p \in \N$, we set
$$\varphi_p(D):=\underset{k \in \Z/n}{\bigcup}\bigl[\{k\} \times F_k^{-1}\{p\}\bigr].$$
Obviously, the family $(\varphi_p(D))_{p \in \N}$ is a partition of $D$.
Now, we fix $p \in \N$ and we prove that $\varphi_p(D)$ has the same deficiency as $D$.
Clearly, the deficiency of $D$ is included in the one of $\varphi_p(D)$. Conversely, let $\beta$ be a limit ordinal in $\omega_1$
and $k \in \Z/n$ be such that $(k,\beta) \not\in \defi(D)$. Let $\alpha<\beta$ be an ordinal.
Then, there are infinitely many ordinals $\gamma \in \left)\alpha,\beta\right($ such that $(k,\gamma) \in D$,
and from the above construction of $\varphi_p(D)$ we conclude that there is an ordinal
$\gamma \in \left)\alpha,\beta\right($ such that $(k,\gamma) \in \varphi_p(D)$. Hence, $(k,\beta) \not\in \defi(\varphi_p(D))$.

We finish by proving that $\Ubd \varphi_p(D)=\Ubd D$. Obviously $\Ubd \varphi_p(D) \leq \Ubd D$.
There exists $k \in \Z/n$ such that $\Ubd D$ is the upper-bound of
$D_k$. Let $\delta <\Ubd D$ be an ordinal.
Since $\Ubd D_k=\Ubd D$ and $\Ubd D$ is a limit ordinal, there are infinitely many ordinals in $)\delta,\Ubd D( \cap  D_k$,
and hence at least one of them is mapped to $p$ by $F_k$, which yields that $\Ubd \varphi_p(D) \geq \delta$.
Varying $\delta$ yields $\Ubd \varphi_p(D) \geq \Ubd D$, and we conclude that $\Ubd \varphi_p(D)=\Ubd D$.
\end{proof}

\begin{proof}[Proof of the Second Partitioning Lemma]
Let $D \in \mathcal{D}$, whose upper-bound we denote by $\gamma$.

First of all, we introduce an elementary splitting method. Let $y \in D$.
We shall construct a partition $D=D_1 \cup D_2$
in which $D_1$ et $D_2$ are admissible, the second projection of $y$ is the greatest element of $\CSup(D_1)$, and
$\defi(D_2)=\defi(D)$.

Then, $\Ubd(D_2)=\gamma$ and the second projection of $D_2$ has no greatest element because neither does the one of $D$.

If the second projection of $y$ is finite, we simply put $D_1:=\{y\}$ and $D_2:=D \setminus \{y\}$. Checking the claimed properties is then straightforward.

Assume now that the second projection of $y$ is infinite, and write it $\delta+j$ for some limit ordinal $\delta$ and some integer $j \geq 0$.
We shall see that $D':=D \cap \bigl(\Z/n \times \left(0,\delta\right(\bigr)$ is admissible with upper-bound $\delta$.
To see this, note first that since $(l,\delta+j) \in D$ for some $l \in \Z/n$ we know from the admissibility of $D$ that for every ordinal $\alpha<\delta$,
    there is an ordinal $\beta \in \left)\alpha,\delta\right( \cap \CSup(D)$, and hence $\beta \in \CSup(D')$.
    This shows that $\Ubd(D')=\delta$.
    Next, we check that $D'$ is admissible. Let $\delta'$ be a limit ordinal. If $\delta'>\delta$, then $(l,\delta'+i)\not\in D$ for all $l \in \Z/n$ and all $i \in \N$.
    Now, if $\delta' \leq \delta$, the construction of $D'$ shows that, for all $k \in \Z/n$, if $(k,\delta')$
    belongs to the deficiency of $D'$, then it belongs to the one of $D$, and hence $(k+i+1,\delta'+i) \not\in D$ for all $i \in \N$, leading to
    $(k+i+1,\delta'+i) \not\in D'$ for all $i \in \N$. Hence, $D'$ is admissible, as claimed.

   Then, with the notation from the First Partitioning Lemma, we set $D_1:=\varphi_0(D') \cup \{y\}$ and $D_2:=D \setminus D_1$.
   Let us check that both $D_1$ and $D_2$ are admissible and that $\defi(D_2)=\defi(D)$.

   Let $(k,\delta') \in \defi(D_1)$. If $\delta'>\delta$ then $(l,\delta'+i)\not\in D_1$ for all $l \in \Z/n$ and all $i \in \N$.
  If $\delta' \leq \delta$, then $(k,\delta')$ belongs to $\defi(\varphi_0(D'))=\defi(D')$, and hence
  $(k,\delta') \in \defi(D)$ (because of the very definition of $D'$); since $D$ is admissible, this leads to $(k+1+i,\delta'+i)\not\in D$ for all $i \in \N$, and finally to
  $(k+1+i,\delta'+i)\not\in D_1$ for all $i \in \N$. Therefore, $D_1$ is admissible.

Next, we prove that $\defi(D_2)=\defi(D)$; since $D_2 \subset D$ and $D$ is admissible, this will prove that $D_2$ is admissible.
The inclusion $\defi(D) \subset \defi(D_2)$ is obvious. Now, let $(k,\delta') \in \defi(D_2)$. If $\delta'>\delta$ then
there is an ordinal $\alpha>\delta+j$ such that there is no $\beta \in \left)\alpha,\delta'\right($ for which $(k,\beta)\in D_2$,
and hence there is no  $\beta \in \left)\alpha,\delta'\right($ for which $(k,\beta)\in D$. Hence,
$(k,\delta') \in \defi(D)$.
Next, assume that $\delta'\leq \delta$. Hence, $(k,\delta')$ belongs to the deficiency of $D' \setminus \varphi_0(D')=\underset{n \geq 1}{\bigcup} \varphi_n(D')$,
a set which has the same deficiency as $D'$ because $\varphi_1(D') \subset \underset{n \geq 1}{\bigcup} \varphi_n(D') \subset D'$
and $\defi\bigl(\varphi_1(D')\bigr)=\defi(D')$; we conclude that $(k,\delta')\in \defi(D)$.

We conclude that $D_2$ is admissible and has the same deficiency as $D$.

\vskip 3mm
Since $D$ is infinite countable, we can choose a bijection
$x : i \in \N \mapsto x_i \in D$.

Next, by induction, we construct a finite sequence of pairwise disjoint admissible subsets of $D$ as follows.
\begin{itemize}
\item First of all, we find an admissible subset $D_0 \subset D$ in $\mathcal{D'} \setminus \mathcal{D}$, that contains $x_0$,
with the second projection of $x_0$ being the greatest element of $\CSup(D_0)$, and such that $D \setminus D_0$ belongs to $\mathcal{D}$ and has the same deficiency as $D$.

\item Let $p \in \N$ and $(D_0,\dots,D_p)$ be a list of pairwise disjoint admissible subsets of $D$, such that
$\CSup(D_i)$ has a greatest element $\beta_i$ for every $i \in \lcro 0,p\rcro$, the set
$D':=D \setminus (D_0 \cup \cdots \cup D_p)$ belongs to $\mathcal{D}$ and has the same deficiency as $D$, and $\{x_0,\dots,x_p\} \subset D_0 \cup \cdots \cup D_p$.
Write $x_{p+1}=(k,\alpha_{p+1})$.
If $(k,\gamma)\not\in \defi(D)$, then there is an ordinal $\beta\geq \alpha_{p+1}$ such that $(k,\beta)\in D'$,
and we take the least such ordinal $\beta_{p+1}$ (so that $\beta_{p+1}=\alpha_{p+1}$ if $x_{p+1} \in D'$). Then, by applying the above construction to $y:=(k,\beta_{p+1})$ and to $D'$,
we obtain an admissible subset $D_{p+1} \subset D'$ that belongs to $\mathcal{D}' \setminus \mathcal{D}$, contains $y$, and is such that $\beta_{p+1}=\Ubd D_{p+1}$ and $\defi(D' \setminus D_{p+1})=\defi(D')=\defi(D)$.
If, on the other hand, $(k,\gamma)\in \defi(D)$, we choose the least integer $N$ such that $x_N \in D'$,
and we simply apply the above construction to $y:=x_N$ and to $D'$.
In any case, note that $x_{p+1} \in D_0 \cup \cdots \cup D_{p+1}$
and that, if $(k,\gamma) \not\in \defi(D)$, then $\CSup(D_{p+1})$ has a greatest element $\beta \geq \alpha_{p+1}$ such that
$(k,\beta)\in D_{p+1}$.
\end{itemize}
This construction yields a sequence $(D_p)_{p \in \N} \in (\mathcal{D}' \setminus \mathcal{D})^\N$ of pairwise disjoint admissible subsets of $D$ such that, for all
$p \in \N$, $x_p \in D_0 \cup \cdots \cup D_p$ and, if $x_p=(k,\alpha_p)$ and $(k,\gamma) \not\in \defi(D)$, then $\CSup(D_{p})$ has a greatest element $\beta \geq \alpha_p$ such that
$(k,\beta)\in D_p$. Since $D=\{x_i \mid i \in \N\}$, we obtain $D=\underset{p \in \N}{\bigcup} D_p$.

It remains to prove that $(D_p)_{p \in \N}$ satisfies the definition of a good partition.
Let $k \in \Z/n$ be such that $(k,\gamma) \not\in \defi(D)$.
Let $\alpha <\gamma$ be an ordinal. Then, there is an element of the form $(k,\beta)$ in $D$ with $\beta \in \left)\alpha,\gamma\right($.
Thus, $x_N=(k,\beta)$ for some $N \in \N$. The above shows that $\CSup(D_N)$ has its greatest element $\beta' \in \left(\beta,\gamma\right($
for which $(k,\beta') \in D_N$, and hence $\beta' \in \left)\alpha,\gamma\right($.
We conclude that $(D_p)_{p \in \N}$ is a good partition of $D$.
\end{proof}

We are now ready to prove the Realization Theorem.
The proof is by induction on the supremum of the combined support of the family under consideration.
Let $(m_{k,\alpha})_{(k,\alpha) \in (\Z/n) \times \omega_1}$ be an admissible family of cardinals.

Assume first that $m_{k,\alpha}=0$ for every infinite ordinal $\alpha$ and every $k \in \Z/n$.
This yields a countable set $I$ together with a mapping $t : I \rightarrow \Z/n \times \N^*$ in which, for all
$(k,\alpha)\in \Z/n \times \N$, the set $t^{-1}\{(k,\alpha+1)\}$ has cardinality $m_{k+\alpha,\alpha}$.
Remember the notation $f^{(k,l)}$ for $k \in \Z/n$ and $l \in \N$ (from Section \ref{augmentationSection}).
Set $g:=\underset{i \in I}{\bigvee} f^{(t(i))}$, which is a countable terminal representation of $\calC_n$.
By Proposition \ref{pointedsumprop}, we see that
$$\forall (k,\alpha)\in \Z/n \times \N, \quad
n_{k,\alpha}(g)=\card t^{-1}\{(k-\alpha,\alpha+1)\}=m_{k,\alpha}.$$

In the remainder of the proof, we assume that $m_{k,\alpha}>0$ for some infinite ordinal $\alpha$
and some $k \in \Z/n$. We denote by $D$ the support of $(m_{k,\alpha})_{(k,\alpha) \in (\Z/n) \times \omega_1}$, and we set $\gamma:=\Ubd(D)$.
We split the discussion into two cases, whether $\gamma$ belongs to $\CSup(D)$ or not.

\noindent \textbf{Case 1: $\gamma \not\in \CSup(D)$.} \\
The Second Partitioning Lemma yields a good partition $(D_p)_{p \in \N}$ of $D$.
Given $p \in \N$, we define a family $m^{(p)}$ of cardinals, indexed over $(\Z/n) \times \omega_1$, as
follows:
$$\forall k \in \Z/n, \; \forall \alpha \in \omega_1, \quad m^{(p)}_{k,\alpha}=1_{D_p}(k,\alpha)\, m_{k,\alpha}.$$
Let $p \in \N$. The support of $m^{(p)}$ is obviously $D_p$, and it is admissible with supremum less than $\gamma$. Hence, by induction,
we find a countable terminal representation $f_p$ of $\calC_n$ such that
$$\forall k \in \Z/n, \; \forall \alpha \in \omega_1, \quad  m^{(p)}_{k,\alpha}=n_{k,\alpha}(f_p).$$
As $(D_p)_{p \in \N}$ is a partition of $D$, we find that
$$\forall k \in \Z/n, \; \forall \alpha \in \omega_1, \quad m_{k,\alpha}=\sum_{p \in \N} m_{k,\alpha}^{(p)}.$$
Hence the pointed sum $f:=\underset{p \in \N}{\bigvee} f_p$ is a countable terminal representation of $\calC_n$ that satisfies
$$\forall k \in \Z/n, \; \forall \alpha \in \omega_1, \quad  m_{k,\alpha}=n_{k,\alpha}(f).$$

\noindent \textbf{Case 2: $\gamma \in \CSup(D)$.} \\
We write $\gamma=\delta+N$ for some limit ordinal $\delta$ and some integer $N \geq 0$.
Note that $m_{k,\gamma}>0$ for some $k \in \Z/n$.
Set
$$M:=\sum_{k \in \Z/n}\, \sum_{p \in \N} m_{k,\delta+p} \quad \text{and} \quad
E:=\bigl\{(k,i,j)\in \Z/n \times \N^2: \; j < m_{k,\delta+i}\bigr\},$$
the latter of which is a subset of $\Z/n \times \N^2$ with cardinality $M$.
Using the fact that $D$ is admissible, we proceed just like in the proof of the Second Partitioning Lemma to find that
$$D':=D \cap \bigl(\Z/n \times \left(0,\delta\right(\bigr)$$
is admissible and that $\delta=\Ubd(D')$. Besides, it is easily checked that $\defi(D')=\defi(D)$.

\noindent \textbf{Subcase 2.1: $M$ is infinite.} \\
We choose a bijection $g : \N \overset{\simeq}{\rightarrow} E$,
whose components we denote by $g_1 : \N \rightarrow \Z/n$, $g_2 : \N \rightarrow \N$ and $g_3 : \N \rightarrow \N$,
so that $g(p)=(g_1(p),g_2(p),g_3(p)) \in \Z/n \times \N \times \N$ for all $p \in \N$.
Next, we apply the First Partitioning Lemma to $D'$.
For each $p \in \N$, we define a family $m^{(p)}$ as follows:
$m^{(p)}_{k,\alpha}:=m_{k,\alpha}$ for all $(k,\alpha) \in \varphi_p(D')$,
$m_{g_1(p),\delta+g_2(p)}^{(p)}=1$, and $m^{(p)}_{k,\alpha}=0$ for all the \emph{remaining} pairs $(k,\alpha) \in \Z/n \times \omega_1$.
Hence,
$$\forall k \in \Z/n, \; \forall \alpha \in \omega_1, \quad m_{k,\alpha}=\sum_{p \in \N} m_{k,\alpha.}^{(p)}$$
Next, let $p \in \N$. We shall construct a countable terminal representation $f^{(p)}$ such that
$$\forall (k,\alpha) \in \Z/n \times \omega_1, \quad n_{k,\alpha}(f^{(p)})=m_{k,\alpha}^{(p)}\, 1_{\varphi_p(D')}(k,\alpha).$$
To this end, we take a good partition $(D_{p,q})_{q \in \N}$ of $\varphi_p(D')$ (given by the Second Partitioning Lemma).
For each integer $q \in \N$, we use the induction hypothesis to recover a countable terminal representation
$f^{(p,q)}=(f_k^{(p,q)},X_k^{(p,q)},o_k^{(p,q)})_{k \in \Z/n}$ such that
$$\forall (k,\alpha) \in \Z/n \times \omega_1, \; n_{k,\alpha}(f^{(p,q)})=m_{k,\alpha}^{(p)}\, 1_{D_{p,q}}(k,\alpha).$$
Next, we consider the pointed sum
$$h_p:=\underset{q \in \N}{\bigvee} f^{(p,q)}=\bigl(f_k^{(p)},X_k^{(p)},o_k^{(p)}\bigr)_{k \in \Z/n,}$$
so that
$$\ell(h_p)=\delta \quad \text{and} \quad
\forall (k,\alpha) \in \Z/n \times \omega_1, \; n_{k,\alpha}(h_p)=m_{k,\alpha}^{(p)}\, 1_{\varphi_p(D')}(k,\alpha).$$
Since $(g_1(p),\delta+g_2(p)) \in D$, we see from the admissibility of $D$ that
$(g_1(p)-g_2(p)-1,\delta) \not\in \defi(D)$, and hence $(g_1(p)-g_2(p)-1,\delta) \not\in \defi(D')$.
Then,  we set $l:=g_1(p)-g_2(p)$ and we check that $h_p$ satisfies condition $(\Aprop_l)$ from Section \ref{augmentationSection}.
Let indeed $\alpha<\delta$ be an ordinal.
By the definition of a good partition, there exists an integer $q \in \N$
and an element $\beta \in \left)\alpha,\delta\right($ such that $\Ubd(D_{p,q})=\beta$ and $(l-1,\beta)\in D_{p,q}$.
In particular, $f^{(p,q)}$ has length $\beta+1$, to the effect that $X^{(p,q)}_{l,\beta+1}=\{o_l^{(p,q)}\}$.
As $n_{l-1,\beta}(f^{(p,q)})>0$, we see that
$X^{(p,q)}_{l-1,\beta}$ is not reduced to $o_{l-1}^{(p,q)}$. Hence,
$f^{(p,q)}_{l-1}$ maps at least one element of $X^{(p,q)}_{l-1,\alpha} \setminus \{o_{l-1}^{(p,q)}\}$ to $o_l^{(p,q)}$, and it follows that
$h_p$ maps at least one element of $X^{(p)}_{l-1,\alpha} \setminus \{o_{l-1}^{(p)}\}$ to $o_l^{(p)}$.

Hence $h_p$ satisfies condition $(\Aprop_l)$ in the definition of the augmentation.
We can therefore define $(h_p)^{+l}$, a countable terminal representation of length $\delta+1$,
and we have $n_{l,\delta}((h_p)^{+l})=1$, $n_{k,\delta}((h_p)^{+l})=0$ for all $k \in \Z/n \setminus \{l\}$, and
$n_{k,\alpha}(h_p^{+l})=n_{k,\alpha}(h_p)$ for all $(k,\alpha)\in \Z/n \times \omega_1$ with $\alpha \neq \delta$.
From there, we can, with the help of Remark \ref{augmentationremark}, continue augmenting at later stages:
by induction, we obtain a sequence $(g^{(p,i)})_{i \in \N}$ of countable terminal representations of $\mathcal{C}_n$, in which
$$g^{(p,0)}=h_p^{+l} \quad \text{and} \quad \forall i \in \N, \; g^{(p,i+1)}=(g^{(p,i)})^{+(l+i+1)}$$
and
$$\forall (k,\alpha) \in (\Z/n) \times \omega_1, \; \forall i\in \N, \quad \begin{cases}
n_{k,\alpha}(g^{(p,i)})=1 & \text{if $\alpha=\delta+i$ and $k=l+i$} \\
n_{k,\alpha}(g^{(p,i)})=n_{k,\alpha}(h_p) & \text{otherwise.}
\end{cases}$$
In particular,
$$\forall (k,\alpha) \in (\Z/n) \times \omega_1, \; n_{k,\alpha}\bigl(g^{(p,g_2(p))}\bigr)=m_{k,\alpha}^{(p)}.$$
We conclude that the pointed sum $g:=\underset{p \in \N}{\bigvee} g^{(p,g_2(p))}$ is a countable terminal representation of $\mathcal{C}_n$ such that
$$\forall (k,\alpha) \in (\Z/n) \times \omega_1, \; n_{k,\alpha}(g)=m_{k,\alpha}.$$

\noindent \textbf{Subcase 2.2: $M$ is finite.} \\
We choose a bijection $g : \lcro 1,M\rcro \overset{\simeq}{\rightarrow} E$.
For all $p \in \{1,\dots,M-1\}$, we consider the set $D'_p:=\varphi_{p-1}(D')$, and we put
$D'_M:=\bigcup_{p \geq M-1} \varphi_p(D')$, which defines a finite partition of $D'$
into admissible subsets, all with supremum $\delta$ and with the same deficiency as $D'$.
Then, the result is obtained by applying the same method as in Subcase 2.1, only with finitely many sequences
$m^{(p)}$.

This completes the proof of the Realization Theorem.

\subsection{Proof of the Adapted Basis Theorem}\label{adaptedProofSection}

We conclude the article by proving the Adapted Basis Theorem.
Let us start from a locally nilpotent linear representation $u$ of $\calC_n$
with countable dimension.

Assume first that $u$ is reduced.
By Lemma \ref{supportconstraintlemma}, $(\kappa_{k,\alpha}(u))_{k \in \Z/n,\alpha\in \omega_1}$ is an admissible family of cardinals.
Hence, by Theorem \ref{realizationtheo}, there exists a countable terminal representation $f$ of $\calC_n$ such that
$n_{k,\alpha}(f)=m_{k,\alpha}$ for every $k \in \Z/n$ and every ordinal $\alpha\in \omega_1$.
Denote by $\widetilde{f}$ its linear realization. Then,
$$\forall (k,\alpha) \in \Z/n \times \omega_1, \quad \kappa_{k,\alpha}(\widetilde{f})=n_{k,\alpha}(f)=\kappa_{k,\alpha}(u).$$
Since both $u$ and $\widetilde{f}$ are reduced and locally nilpotent with countable dimension, the Cyclic Kaplansky Theorem
yields that they are isomorphic linear representations of $\calC_n$. However, it is clear that
$\widetilde{f}$ has an adapted basis: indeed if we write $f=(X_k,o_k,g_k)_{k \in \Z/n}$,
then the families consisting of the standard basis of $\F^{(X_k \setminus \{o_k\})}$, for $k \in \Z/n$, qualify by the very definition of
$\widetilde{f}$. Hence, $u$ has an adapted basis.

Let us now consider the general case. We have seen in Section \ref{extractsaturatedSection} that $u$ is isomorphic to the direct sum of $u_{\red}$ and $u_{\infty}$, which are reduced and saturated locally nilpotent representations of $u$, respectively.
Obviously, $u_{\red}$ has countable dimension. Hence, by the above special case, $u_{\red}$ has an adapted basis.
On the other hand, $u_{\infty}$ has an adapted basis by Theorem \ref{saturatedadaptedbasisTheo}. It is clear from there that $u_{\red} \oplus u_{\infty}$
has an adapted basis, and we conclude that $u \simeq u_{\red} \oplus u_{\infty}$ has an adapted basis.


\begin{thebibliography}{1}
\bibitem{DeOliveira}
D. Duarte de Oliveira, V. Futorny, T. Klimtchuk, D. Kovalenko and V.V. Sergeichuk,
{Cycles of linear and semilinear mappings.}
Linear Algebra Appl.
{\bf 438} (2013), 3442--3453.


\bibitem{Fuchs}
L. Fuchs,
{\em Infinite Abelian Groups, Vol. 1,}
Pure and Applied Mathematics \textbf{36},
Academic Press, 1970.

\bibitem{Gelonch}
J. Gelonch,
{The contragredient equivalence for several matrices: a set of invariants.}
Linear Algebra Appl.
{\bf 291} (1999), 37--49.


\bibitem{Kaplansky}
I. Kaplansky,
{\em Infinite Abelian Groups,}
University of Michigan Press,
Ann Arbor, 1954.

\bibitem{Nazarova}
L.A. Nazarova,
{Representations of quivers of infinite type.}
Math. USSR Izv.
{\bf 7} (1973), 749--792.

\bibitem{Sergeichuk}
V.V. Sergeichuk,
{Computation of canonical matrices for chains and cycles of linear mappings.}
Linear Algebra Appl.
{\bf 376} (2004), 235--263.



\end{thebibliography}
\end{document}